\documentclass[a4paper, 11pt, final]{article}

\usepackage[all]{xy}
\usepackage{style}
 \usepackage{amsthm}
\usepackage{dsfont}

 \usepackage{geometry}

\newcommand{\CH}{\operatorname{CH}\nolimits}

 \title{Chow--K\"unneth decomposition for $3$- and $4$-folds fibred by
   varieties with trivial Chow group of zero-cycles}

\author{Charles Vial}
\date{}

\begin{document}

\maketitle

\begin{abstract} Let $k$ be a field and let $\Omega$ be a universal
  domain over $k$.  Let $f:X \r S$ be a dominant morphism defined over
  $k$ from a smooth projective variety $X$ to a smooth projective
  variety $S$ of dimension $\leq 2$ such that the general fibre of
  $f_\Omega$ has trivial Chow group of zero-cycles. For example, $X$
  could be the total space of a two-dimensional family of varieties
  whose general member is rationally connected. Suppose that $X$ has
  dimension $\leq 4$. Then we prove that $X$ has a self-dual Murre
  decomposition, i.e.  that $X$ has a self-dual Chow--K\"unneth
  decomposition which satisfies Murre's conjectures (B) and (D).
  Moreover we prove that the motivic Lefschetz conjecture holds for
  $X$ and hence so does the Lefschetz standard conjecture.  We also
  give new examples of threefolds of general type which are Kimura
  finite-dimensional, new examples of fourfolds of general type having
  a self-dual Murre decomposition, as well as new examples of
  varieties with finite degree three unramified cohomology.
\end{abstract}

\section*{Introduction}

Throughout this paper, algebraic cycles and Chow groups are with
rational coefficients.  Let $X$ be a smooth projective complex
variety. The Hodge conjecture predicts that every Hodge class in
$H_{2i}(X) := H_{2i}(X(\C),\Q)$ is the class of an algebraic cycle. In
particular, given any two smooth projective complex varieties $X$ and
$Y$, the Hodge conjecture predicts that any morphism $f$ of Hodge
structures between $H_*(X)$ and $H_*(Y)$ comes from geometry.  By this
we mean that $f$ is induced by a correspondence between $X$ and $Y$,
that is, by an algebraic cycle on $X \times Y$. Whether the Hodge
conjecture happens to be true or not, Grothendieck pointed out that
certain morphisms of Hodge structures play a more important role in
the theory of algebraic cycles. If $X$ has pure dimension $d$, he
suggested that for all integer $i$ the K\"unneth component in
$H_{2d}(X \times X)$ inducing the projector on $H_i(X)$ should be
induced by a correspondence. He also suggested that, for all $i\leq
d$, the inverse to the Lefschetz isomorphism $H_{2d-i}(X) \r H_i(X)$
given by intersecting $d-i$ times with the class of a smooth
hyperplane section should be induced by an algebraic cycle.  The first
conjecture is usually referred to as the K\"unneth standard conjecture
and the second one to the Lefschetz standard conjecture.  Classically
\cite[4.1]{Kleiman}, it is known that the Lefschetz standard
conjecture for $X$ implies the K\"unneth standard conjecture for $X$.

If the Hodge conjecture gives a simple description of the image of the
cycle class map $\CH_i(X) \r H_{2i}(X)$ in terms of the Hodge structure
of $H_{2i}(X)$, it is a much more difficult problem to unravel the
nature of its kernel. Beilinson and Bloch, inspired by Grothendieck's
philosophy of motives, first proposed a description of such a kernel
in terms of a descending filtration on Chow groups that would behave
functorially with respect to the action of correspondences and would
be such that its graded parts would depend solely on the homological
motive of $X$.

More generally, the conjectures of Bloch and Beilinson can be
formulated for smooth projective varieties defined over any field $k$
if one uses $\ell$-adic cohomology in place of Betti cohomology. In
that setting, for $X$ smooth projective over $k$, the K\"unneth and
Lefschetz standard conjectures stipulate the existence of algebraic
cycles in $X \times X$ inducing the right action on cohomology. This
is consistent with the previous formulations. Indeed, let $X$ be a
smooth projective variety defined over a subfield $k \subseteq \C$ and
assume that there is a cycle $\Gamma \in \CH^{i}(X_\C \times_\C X_\C)$
inducing the inverse to the Lefschetz isomorphism $H_{2d-i}(X) \r
H_i(X)$. Here, $d$ is the dimension of $X$ and $i$ is a non-negative
integer $\leq d$. It is a fact that, for a smooth projective variety
$Y$ defined over $k$, the base change map $\CH_*(Y_K) \r \CH_*(Y_L)$ is
an isomorphism modulo homological equivalence for all extensions of
algebraically closed fields $L/K$ over $k$. Therefore, up to replacing
$\Gamma$ by a cycle homologically equivalent to it, we can assume that
$\Gamma$ is defined over the algebraic closure of $k$ inside $\C$.
But then, $\Gamma$ is defined over a finite extension of $k$ which can be
chosen to be Galois, say of degree $n$, and with Galois group $G$. It
is then straightforward to check that the cycle $\gamma := 1/n \cdot
\sum_{\sigma \in G} \sigma_*\Gamma$ is defined over $k$ and induces
the inverse to the Lefschetz isomorphism $H_{2d-i}(X) \r H_i(X)$. The
same arguments apply to the K\"unneth standard conjecture.

Twenty years ago, Murre proposed that not only should the K\"unneth
projectors in cohomology be induced by correspondences, but also that
they should be induced by correspondences that are idempotents modulo
rational equivalence. Given $X$ a smooth projective variety of
dimension $d$ over a field $k$, Murre \cite{Murre1} conjectured the
following. \medskip

(A) $X$ has a Chow--K\"unneth decomposition $\{\pi_0, \ldots,
\pi_{2d}\}$ : There exist mutually orthogonal idempotents $\pi_0,
\ldots, \pi_{2d} \in \CH_d(X \times X)$ adding to the identity such
that $(\pi_i)_*H_*(X)=H_i(X)$ for all $i$.

(B) $\pi_0, \ldots, \pi_{2l-1},\pi_{d+l+1}, \ldots, \pi_{2d}$ act
trivially on $\CH_l(X)$ for all $l$.

(C) $F^i\CH_l(X) := \ker(\pi_{2l}) \cap \ldots \cap \ker(\pi_{2l+i-1})$
does not depend on the choice of the $\pi_j$'s. Here the $\pi_j$'s are
acting on $\CH_l(X)$.

(D) $F^1\CH_l(X) = \CH_l(X)_\hom$, where the subscript `hom' refers to
homologically trivial cycles. \medskip

A variety $X$ which has a Chow--K\"unneth decomposition that satisfies
Murre's conjectures (B) and (D) is said to have a \emph{Murre
  decomposition}. If moreover the Chow--K\"unneth decomposition of
conjecture (A) can be chosen so that $\pi_i = {}^t\pi_{2d-i} \in
\CH_d(X \times X)$, then $X$ is said to have a \emph{self-dual Murre
  decomposition}. Here, we understand Murre's conjecture (C) as saying
that any two filtrations induced by two distinct Chow--K\"unneth
decompositions for $X$ coincide, not that they are merely isomorphic.

The relevance of Murre's conjectures was demonstrated by Jannsen
\cite{Jannsen} who showed that these hold for all smooth projective
varieties if and only if Bloch's and Beilinson's conjectures hold for
all smooth projective varieties. Murre's formulation of a conjectural
descending filtration on Chow groups has the advantage over
Beilinson's and Bloch's that it does not involve any functoriality
properties and that it can therefore be proven on a case-by-case
basis. Since Murre's paper \cite{Murre1} appeared, many authors have
tried to prove those conjectures for certain classes of varieties. In
this paper, we extend the list of cases for which these can be proven.
\medskip

Let $X$ be a smooth projective variety defined over a field $k$ and
let $\Omega$ be a universal domain over $k$, i.e.  $\Omega$ is an
algebraically closed field of infinite transcendence degree over $k$.
A smooth projective variety $X$ will be said to have trivial Chow
group of zero-cycles if $\CH_0(X_\Omega) = \Q$.  Our main result is the
following

\begin{theorem2} \label{mainth} Let $f:X \r S$ be a dominant morphism
  between smooth projective varieties defined over a field $k$ such
  that the general fibre of $f_\Omega$ has trivial Chow group of
  zero-cycles.  Suppose that $S$ has dimension $\leq 2$ and that $X$
  has dimension $\leq 4$.  Then $X$ has a self-dual Murre
  decomposition.  Moreover, the motivic Lefschetz conjecture, as stated
  in \S\ref{motLef}, holds for $X$ and hence so does the Lefschetz
  standard conjecture.
 \end{theorem2}

 Let's stress that the theorem gives a self-dual Murre decomposition
 of $X$ which is defined over a field of definition of $f$.
 Together
 with standard results on rationally connected varieties
 \cite[IV.3.11]{Kollar} (see also the proof of Corollary
 \ref{ratconncoro}), we deduce

 \begin{theorem2} \label{ratconn2}
   Let $f:X \r S$ be an equidimensional dominant morphism between
   smooth projective varieties defined over a field $k$ whose general
   fibre is separably rationally connected.  Suppose that $S$ has
   dimension $\leq 2$ and that $X$ has dimension $\leq 4$.  Then $X$
   has a self-dual Murre decomposition which satisfies the motivic
   Lefschetz conjecture.
 \end{theorem2}

 Theorem \ref{mainth} contrasts with the approach of
 Gordon--Hanamura--Murre \cite{GHM} where Chow--K\"unneth decompositions
 are constructed for varieties $X$ that come with a fibration $f : X
 \r S$ which is ``nice'' enough: it is assumed, among other things, in
 \emph{loc.  cit.} that $f$ should be smooth away from a finite number
 of points on $S$ and that $f$ should have a relative Chow--K\"unneth
 decomposition. Here, we do not even require $f$ to be flat.

 Theorem \ref{mainth} was already proved in the case $\dim S \leq 1$ :
 a self-dual Chow--K\"unneth decomposition for which the motivic
 Lefschetz conjecture holds was constructed in \cite[4.6]{Vial3} and
 Murre's conjectures were checked to hold in \cite[4.21]{Vial2}. The
 results in \cite{Vial2, Vial3} just mentioned are more generally
 valid for fourfolds with Chow group of zero-cycles supported on a
 curve. By Theorem \ref{isogeny}, it is the case that, if $X$ is as in
 Theorem \ref{mainth} with $\dim S \leq 1$, then $\CH_0(X_\Omega)$ is
 supported on a curve. From now on, we will therefore focus on the
 case $\dim S =2$.

 Here, Murre's conjecture (C) is proved only on the grounds that the
 idempotents of a Chow--K\"unneth decomposition for $X$ are supported
 in a specific dimension, cf.  section \ref{Murreconj}.  It is however
 fully proved under some extra assumption on $X$ in Theorem
 \ref{Murre-theorem} and Theorem \ref{C-Murre2}. Let's mention that
 del Angel and M\"uller-Stach \cite{dAMS} proved the existence of a
 Murre decomposition for threefolds fibred by conics over a surface
 (see also the recent paper \cite{MSS} where Chow--K\"unneth
 decompositions are constructed for some threefolds including conic
 fibrations). In addition to treating the four-dimensional case, our
 theorem makes more precise the result of \cite{dAMS} by showing that
 the Murre decomposition can be chosen to be self-dual and by showing
 the motivic Lefschetz conjecture for $X$.  Also our approach is
 different from \cite{dAMS}. Del Angel and M\"uller-Stach assume that
 all the fibres of $f$ are rationally connected, this allows them to
 compute the cohomology of $X$ via the Leray spectral sequence. They
 then construct idempotents modulo rational equivalence and check that
 they act as the K\"unneth projectors on cohomology. Here, we do not
 make any assumptions on the bad fibres of $f$ and we first compute
 the Chow group of zero-cycles of $X$ to only then deduce that the
 idempotents we construct act as the K\"unneth projectors on
 cohomology.  \medskip

 A word about the proof of the theorem and about the organisation of
 the paper. In section \ref{geom}, we make the simple observation that,
 for $f$ as in Theorem \ref{mainth}, $f_* : \CH_0(X_\Omega) \r
 \CH_0(S_\Omega)$ is bijective with inverse induced by an algebraic
 correspondence which is defined over a field of definition of $f$.
 Together with Theorem \ref{effective}, the proof of the existence of
 a Murre decomposition for $X$ essentially reduces to the case of
 motives of surfaces. The validity of Murre's conjectures (A), (B) and
 (D) for surfaces goes back to Murre himself \cite{Murre}.  However,
 by the very nature of Theorem \ref{effective}, we need Murre's
 conjectures not only for surfaces but for motives of surfaces (i.e.
 we need to deal with idempotents).  This is the object of section
 \ref{AP}. The construction of idempotents inducing the right
 K\"unneth projectors for $X$ in homology is carried out in section
 \ref{CK}. Constructing such idempotents is easy from the case of
 surfaces. However, these are not necessarily mutually orthogonal. The
 non-commutative Gram--Schmidt process which already appears in
 \cite{Vial3} and which is run on this set of idempotents is described
 in Lemma \ref{linalg}. This way, we obtain a \emph{self-dual}
 Chow--K\"unneth decomposition for $X$. The motivic Lefschetz
 conjecture is formulated in section \ref{motLef}, its relevance is
 discussed and it is proved for $X$ there.  Murre's conjectures (B)
 and (D) are then proved for $X$ in section \ref{Murreconj} by using
 results of \cite{Vial2} which are recalled in Proposition
 \ref{MurreB} and \ref{MurreC}.  If $C$ is a smooth projective curve,
 then the results of \emph{loc. cit.}  actually make it possible to prove
 Murre's conjectures (B) and (D) for $X \times C$. This is the object
 of section \ref{Murreconj2}.

 Murre's observation that the K\"unneth projectors should lift to
 idempotents modulo rational equivalence is crucial in the sense that
 a combination of Beilinson's and Bloch's conjectures with
 Grothendieck's standard conjectures imply that any projector in
 homology should be liftable to an idempotent modulo rational
 equivalence. Shun-Ichi Kimura \cite{Kimura} introduced a notion of
 finite-dimensionality for Chow motives which implies such a lifting
 property for projectors. Kimura's notion of finite-dimensionality has
 become widely popular for this reason and more importantly because of
 its relationship to Murre's conjectures.  The simple observation of
 Theorem \ref{isogeny} is used in section \ref{fdprob} to give new
 examples of threefolds of general type which are Kimura
 finite-dimensional, namely threefolds fibred by Godeaux surfaces.
 There, using the result of section \ref{CK}, we also show in Theorem
 \ref{conic-Kimura} that, if $X$ is a conic fibration over a surface
 which is Kimura finite-dimensional, then $X$ is Kimura
 finite-dimensional.  In section \ref{general}, we produce examples of
 fourfolds of general type with Chow group of zero-cycles not
 supported on a curve but which admit a self-dual Murre decomposition.
 Such examples will be given by fourfolds fibred by surfaces
 birational to Godeaux surfaces.  Theorem \ref{isogeny} is slightly
 generalised in Theorem \ref{repsurface} ; this is used in section
 \ref{unram} to prove finiteness of unramified cohomology in some new
 cases. \medskip

 Finally, although we don't state it here, the methods of this paper
 actually show that Murre's conjectures hold for smooth projective
 fourfolds $X$ for which there exist a smooth projective surface $S$
 and correspondences $\alpha \in \CH_2(S \times X)$ and $\beta \in
 \CH^2(X \times S)$ such that $\beta \circ \alpha = \Delta_S \in \CH_2(S
 \times S)$ and such that $\alpha \circ \beta$ acts as the identity on
 $\CH_0(X_\Omega)$. Consequently, if $f : X \r S $ is a dominant
 morphism to a surface $S$ such that the general fibre of $f_\Omega$
 has trivial Chow group of zero-cycles, then Murre's conjectures hold
 for any smooth projective variety $X'$ which is birational to $X$.

 \paragraph{Notations.} We refer to \cite{Scholl} for the notion of
 Chow motive and to \cite{KMP} for the covariant notations we use
 here.  Briefly, a Chow motive $M$ is a triple $(X,p,n)$ where $X$ is
 a smooth projective variety over $k$ of pure dimension $d$, $p \in
 \CH_d(X\times X)$ is an idempotent ($p\circ p = p$) and $n$ is an
 integer. The motive $M$ is said to be effective if $n \geq 0$. The
 dual of $M$ is the motive $M^\vee := (X,{}^tp,-d-n)$, where ${}^tp$
 denotes the transpose of $p$. A morphism between two motives
 $(X,p,n)$ and $(Y,q,m)$ is a correspondence in $q \circ \CH_{d+n-m} (X
 \times Y) \circ p$. We write $\h(X)$ for the motive of $X$, i.e. for
 the motive $(X,\Delta_X,0)$ where $\Delta_X$ is the class of the
 diagonal inside $\CH_d(X \times X)$. We have $\CH_i(X,p,n) =
 p_*\CH_{i-n}(X)$ and $H_i(X,p,n) = p_*H_{i-2n}(X)$, where we write
 $H_i(X) := H^{2d-i}(X(\C),\Q)$ for Betti homology when $k \subseteq
 \C$. The results of this paper are valid more generally for any field
 $k$ if one considers $\ell$-adic homology $H_i(X,\Q_\ell) :=
 H^{2d-i}(X_{\bar{k}},\Q_\ell)$ with $\ell \neq \mathrm{char} \ k$ in
 place of Betti homology. In that case the action of a correspondence
 $\Gamma \in \CH_i(X \times Y)$ on $H_i(X,\Q_\ell)$ is given by the
 action of $\Gamma \otimes_\Q 1 \in \CH_i(X \times Y) \otimes_\Q
 \Q_\ell$.

 \paragraph{Acknowledgements.} I am grateful to Sergey Gorchinskiy for
 his interest and for providing me with a simpler construction of the
 Albanese projector.  Thanks to the referee for his detailed remarks
 and numerous suggestions.  Thanks to Mingmin Shen for stimulating
 discussions and Burt Totaro for useful comments.  This work is
 supported by a Nevile Research Fellowship at Magdalene College,
 Cambridge and an EPSRC Postdoctoral Fellowship under grant
 EP/H028870/1. I would like to thank both institutions for their
 support.

\section{A geometric result on zero-cycles} \label{geom}

Let $X$ and $S$ be smooth projective varieties over $k$ and let $f : X
\r S$ be a dominant morphism. Then any general linear section $\sigma
: H \rightarrow X$ of dimension $\dim S$ is smooth over $k$ and is
such that the morphism $\pi :=f|_H : H \r S$ is dominant. In
particular, for any general $H$, $\pi$ is generically finite and its
degree is written $n$.

\begin{proposition} \label{surjective} Let $f: X \rightarrow S$ be a
  dominant morphism. Then, for any general $H$ as above, $\Gamma_f
  \circ \Gamma_\sigma \circ {}^t\Gamma_\pi = n\cdot \Delta_S \in
  \CH_{\dim S}(S \times S)$.
 \end{proposition}

\begin{proof} This follows from the projection formula applied to $\pi
  = f \circ \sigma$ and Manin's identity principle. See \cite[Example
  1 p. 450]{Manin}.
 \end{proof}

\begin{definition}
  Let $f: X \rightarrow S$ be a dominant morphism between smooth
  projective varieties defined over a field $k$.  A general point of
  $S$ is a closed point sitting outside a given proper closed subset
  of $S$. By \emph{general fibre} of $f$, we mean the fibre of $f$
  over a general point of $S$.
\end{definition}

\begin{theorem} \label{isogeny} Let $f: X \rightarrow S$ be a dominant
  morphism between smooth projective varieties defined over a field
  $k$.  Assume that a general fibre $Y$ of $f$ satisfies $\CH_0(Y) =
  \Q$.
  Then the induced map $f_* : \CH_0(X) \r \CH_0(S)$ is bijective and
  its inverse is induced by a correspondence $\Gamma \in \CH^{\dim X}(S
  \times X)$. Moreover $\Gamma$ can be chosen to be defined over a
  field of definition of $f$.
 \end{theorem}
 \begin{proof} Let's show that the correspondence $\Gamma$ can be
   chosen to be $\frac{1}{n} \Gamma_\sigma \circ {}^t\Gamma_\pi$.
   According to Proposition \ref{surjective}, it suffices to prove that
   the correspondence $ \Gamma_\sigma \circ {}^t\Gamma_\pi$ induces a
   surjective map $(\Gamma_\sigma \circ {}^t\Gamma_\pi)_* : \CH_0(S) \r
   \CH_0(X)$.

   Let's fix an open subset $U$ of $S$ such that $\pi : H_U \r U$ is
   finite and such that the fibres of $f_U$ satisfy $\CH_0(X_u)=\Q$ for
   all closed points $u$ in $U$.
   Let $p$ be a closed point of $X$ and let $[p]$ denote the class of
   $p$ in $\CH_0(X)$. By Chow's moving lemma, the zero-cycle $[p] \in
   Z_0(X)$ is rationally equivalent to a zero-cycle $\alpha = \sum a_i
   \cdot [p_i]$ supported on $X_U$. This means that each $p_i$ is a
   closed point of $X$ that belongs to the open subset $X_U$ of $X$.
   Let $u_i := f(p_i)$. Since $\CH_0(X_{u_i}) = \Q$ and $\deg
   (\sigma_*\pi^*[u_i]) \neq 0$, we see that $p_i$ is rationally
   equivalent to $\sigma_*\pi^*[u_i]$ taken with an appropriate
   rational coefficient. Thus, $\sigma_*\pi^*$ is surjective on
   zero-cycles.
   \end{proof}

For example, we get as a corollary the following which is used in
\cite[Cor. 4.23]{Vial2} and \cite[Cor. 4.7]{Vial3} in the cases when
$S$ is a curve.

\begin{corollary} \label{ratconncoro} Let $f : X \rightarrow S$ be an
  equidimensional dominant morphism between smooth projective
  varieties defined over a field $k$. Assume that the general fibre of
  $f$ is separably rationally connected (e.g. $X$ could be a Fano
  fibration).  Then $f_*:\CH_0(X_\Omega) \r \CH_0(S_\Omega)$ is
  bijective and there is a correspondence $\Gamma \in \CH^{\dim X}(S
  \times X)$ such that $\Gamma_* : \CH_0(S_\Omega) \r \CH_0(X_\Omega)$
  is the inverse of $f_*$. \qed
\end{corollary}
\begin{proof}
  According to Theorem \ref{isogeny}, it suffices to prove that a
  general fibre of $f_\Omega$ has trivial Chow group of zero-cycles.
  By considering a smooth closed fibre of $f$ which is separably
  rationally connected, we see after pulling back to $\Omega$ that
  $f_\Omega$ has a smooth closed fibre which is separably rationally
  connected. It follows from \cite[IV.3.11]{Kollar} that a general
  fibre of $f_\Omega$ is separably rationally connected and hence that
  a general fibre of $f_\Omega$ has trivial Chow group of zero-cycles.
\end{proof}

Theorem \ref{repsurface} below, which generalises Theorem
\ref{isogeny} is irrelevant to the proof of Theorem \ref{mainth}.
However, we include it here because of the interesting consequences it
has for unramified cohomology, see section \ref{unram}.

\begin{definition}
  A smooth projective variety $X$ over $k$ is said to have
  \emph{representable} Chow group of algebraically trivial $i$-cycles
  if there exists a curve $C$ over $\Omega$ and a correspondence
  $\gamma \in \CH_{i+1}(C \times X_\Omega)$ such that
  $\gamma_*\CH_0(C)_{alg} = \CH_i(X_\Omega)_{alg}$.
\end{definition}

\begin{lemma} \label{rep} Let $X$ be a smooth projective variety over
  $k$. Then the following statements are equivalent.

 $1.$ $\CH_0(X)_{\alg}$ is representable.

 $2.$ The Albanese map $\mathrm{alb}_{X_\Omega} :
  \CH_0(X_\Omega)_{alg} \r \mathrm{Alb}_{X_\Omega}(\Omega)$ is an
  isomorphism (this map is always surjective).

 $3.$ If $\iota :C \r X$ is any smooth linear
  section of $X$ of dimension $1$, the induced map $\iota_* :
  \CH_0(C_\Omega) \r \CH_0(X_\Omega)$ is surjective.
\end{lemma}

\begin{proof} This was proved by Jannsen \cite[1.6]{Jannsen}.
  \end{proof}

\begin{theorem} \label{repsurface} Let $f : X \rightarrow S$ be a
  generically smooth and dominant morphism between smooth projective
  varieties defined over a field $k$. Assume that the general fibre
  $Y$ of $f_\Omega$ is such that $\CH_0(Y)_\alg$ is representable.
  Then $\CH_0(X)$ is supported in dimension $\dim S +1$. This means
  that there exists a smooth projective variety $T$ over $k$ of
  dimension $\dim S +1$ and a correspondence $\Gamma \in \CH^{\dim X}(T
  \times X)$ such that $(\Gamma_\Omega)_* : \CH_0(T_\Omega) \r
  \CH_0(X_\Omega)$ is surjective.
\end{theorem}

\begin{proof} Let $\iota : H \r X$ be a smooth linear section of $X$
  of dimension $\dim S + 1$ such that $f$ restricted to $H$ is
  dominant and generically smooth.  Let $U$ be an open subset of
  $S_\Omega$ such that $f_U : X_U \r U$ is smooth, $f_U|_{H_U} : H_U
  \r U$ is smooth and such that the fibres of $f_U$ have representable
  Chow group of zero-cycles.

  Let's prove that $(\iota_\Omega)_* : \CH_0(H_\Omega) \r
  \CH_0(X_\Omega)$ is surjective. Let $p$ be a closed point of
  $X_\Omega$. By Chow's moving lemma, the zero-cycle $[p]$ is
  rationally equivalent to a cycle $\alpha = \sum a_i \cdot [p_i]$
  supported on $X_U$. Let $s_i := f(p_i) \in U$. Then, by choice of
  $U$, each cycle $[p_i]$ is rationally equivalent on $X_{s_i}$ to a
  cycle $\beta_i$ supported on $H_{s_i}$. Now clearly $\sum a_i \cdot
  \beta_i$ is in the image of $(\iota_\Omega)_* : \CH_0(H_\Omega) \r
  \CH_0(X_\Omega)$ and hence so is $[p]$.
\end{proof}

\begin{remark}
  The descent properties of Theorems \ref{isogeny} \&
  \ref{repsurface}, i.e. the fact that the correspondence $\Gamma$ in
  those theorems can be chosen to be defined over a field of
  definition of $f$, are essential to proving that $X$ has a
  Chow--K\"unneth decomposition defined over the field of definition of
  $f$ (Theorem \ref{theoremCK}) and to proving Proposition
  \ref{bielliptic-unram}.
\end{remark}

\begin{remark} Under the assumptions of the above theorem, it is not
  true that if $\CH_0(S)$ is supported in dimension, say $n$, then
  $\CH_0(X)$ is supported in dimension $n+1$.  Consider for example a
  Lefschetz fibration $S \r \P^1$ associated to a smooth projective
  surface $S$ with non-representable Chow group of zero-cycles.
\end{remark}

\begin{remark} \label{BB} Let $Y$ be a smooth projective variety over
  $k$ and let $H \r Y$ be a smooth linear section of dimension $n$. A
  consequence of the conjectures of Bloch and Beilinson is that, if
  $\CH_0(Y_\Omega)$ is supported in dimension $n$, then $\CH_0(H_\Omega)
  \r \CH_0(Y_\Omega)$ is surjective. Therefore, if one believes in the
  conjectures of Bloch and Beilinson, then the above theorem can be
  extended to the following.  Let $f : X \r S$ be a generically smooth
  and dominant morphism between smooth projective varieties defined
  over a field $k$. Let $n$ be a positive integer and assume that the
  general fibre $Y$ of $f$ is such that $\CH_0(Y)$ is supported in
  dimension $n$.  Then $\CH_0(X)$ is supported in dimension $\dim S +
  n$.
\end{remark}

\section{Effective motives with trivial Chow group of zero-cycles}

The following theorem was mentioned to me by Bruno Kahn. Its proof
uses, among other things, the technique of Bloch and Srinivas
\cite{BS} together with Theorem 2.4.1 of Kahn--Sujatha
\cite{KahnSujatha} where it is shown that a correspondence $\Gamma \in
\CH_d(X \times Y) = \Hom(\h(X),\h(Y))$ which vanishes in $\CH_d(U \times
Y)$ for some open subset $U \subset X$ factors through some motive
$\h(Z)(1)$ with $\dim Z = d -1$.

\begin{theorem}\label{effective}
  Let $M=(X,p)$ be an effective Chow motive such that $\CH_0(M_\Omega)
  = 0$. Then there exist a smooth projective variety $Y$ of
  dimension $\dim X -1$ and an idempotent $q \in \CH_{\dim Y}(Y \times
  Y)$ such that $(X,p,0) \simeq (Y,q,1)$.
\end{theorem}

\begin{proof} Without loss of generality, we can assume that $X$ is
  connected.  Since we are working with rational coefficients, the
  assumption $\CH_0(M_\Omega)=0$ implies $\CH_0(M_{k(X)})=0$ which means
  $p_*\CH_0(X_{k(X)}) = 0$. In particular, if $\eta_X$ denotes the
  generic point of $X$, then we have $p_*\eta_X=0$.  But $p_*\eta_X$
  is the restriction of $p \in \CH_d(X \times X)$ to $\varinjlim \CH_d(U
  \times X) = \CH_0(X_{k(X)})$, where the limit is taken over all open
  subsets $U$ of $X$. Therefore, by the localization exact sequence
  for Chow groups, there exist a proper closed subset $D \subset X$
  and a correspondence $\gamma \in \CH_d(D \times X)$ such that
  $\gamma$ maps to $p$ via the inclusion $D \times X \r X \times X$.
  Up to shrinking the open $U$ for which $p|_{U \times X}$ vanishes,
  we can assume that $D$ has pure dimension $d-1$. Let $ Y \r D$ be an
  alteration of $D$ and let $\sigma : Y \r D \hookrightarrow X$ be the
  composite morphism. The induced map $\CH_d(Y \times X) \r \CH_d(D
  \times X)$ is surjective and we have $p= (\sigma \times \id_X)_* f$,
  where $f \in \CH_d(Y \times X)$ is a lift of $\gamma$. Then, by
  \cite[16.1.1]{Fulton}, we have $(\sigma \times \id_X)_* f = f \circ
  {}^t\Gamma_\sigma$. This yields a factorisation $p = f \circ g$,
  where $f\in \CH_d(Y \times X)$ and $g = {}^t\Gamma_\sigma \in \CH^d(X
  \times Y)$.  Let's consider the correspondence $q:=g \circ f \circ g
  \circ f = g \circ p \circ f \in \CH_{d-1}(Y \times Y)$. It is
  straightforward to check that $q$ is an idempotent, and that $p
  \circ f \circ q \circ g \circ p = p$ as well as $q \circ g \circ p
  \circ f \circ q = q$.  These last two equalities exactly mean that
  $p \circ f \circ q$ seen as a morphism of Chow motives from
  $(Y,q,1)$ to $(X,p,0)$ is an isomorphism with inverse $q \circ g
  \circ p$.
\end{proof}

As noted by Sergey Gorchinskiy \cite{Gor}, this theorem admits the
following corollary

\begin{corollary} \label{effective-coro} Let $m$ and $n$ be positive
  integers.  Let $M=(X,p)$ be an effective Chow motive such that
  $\CH_i(M_\Omega) = 0$ for $i\leq n-1$ and $\CH_j(M_\Omega^\vee(\dim
  X)) = \CH_{j-\dim X}(M^\vee_\Omega) = 0$ for $j\leq m-1$. Then there
  exists a smooth projective variety $Y$ of dimension $\dim X - m - n$
  such that $M$ is isomorphic to a direct summand of $\h(Z)(n)$.
\end{corollary}
\begin{proof}
  By Theorem \ref{effective} applied $n$ times, there is a smooth
  projective variety $Y$ of dimension $\dim X - n$ and an idempotent
  $q \in \CH_{\dim X -n}(Y\times Y)$ such that $M \simeq (Y,q,n)$. We
  then have by duality $M^\vee(\dim X) \simeq (Y,{}^t q)$. Applying
  $m$ times Theorem \ref{effective} gives a smooth projective variety
  $Z$ of dimension $\dim Y - m = \dim X - n - m$ such that
  $M^\vee(\dim X)$ is isomorphic to a direct summand of $\h(Z)(m)$.
  Therefore, after dualizing, we see that $M$ is isomorphic to a
  direct summand of $\h(Z)(n)$.
\end{proof}

\section{The Albanese motive and the Picard motive} \label{AP}

The results of the previous section show that it is convenient not
only to deal with smooth projective varieties but also with
idempotents. It may, however, be difficult to deal with idempotents
because these are usually not central. Here, we extend the construction
of Murre's Albanese projector to the case of Chow motives.\medskip

I thank Sergey Gorchinskiy \cite{Gor} for mentioning the following
basic lemma and the construction that ensues. The difference between
Lemma \ref{triangular} and Lemma \ref{linalg} is that Lemma
\ref{linalg} makes it possible to preserve self-duality when
orthonormalising a family of idempotents.

\begin{lemma}\label{triangular} Let $p$ be an idempotent endomorphism
  of an object $A \oplus B$ in a Karoubi closed additive category. Let
  $p_A$ denote the composition$$A \hookrightarrow A \oplus B
  \stackrel{p}{\longrightarrow} A \oplus B \r A$$ and similarly for
  $p_B$. Assume that $p$ is upper-triangular, that is, the composition
$$A \hookrightarrow A \oplus B \stackrel{p}{\longrightarrow}
A \oplus B \r B$$ vanishes. Then $p_A$ and $p_B$ are idempotents and
there is a canonical isomorphism
$$\im(p) \simeq
\im(p_A) \oplus \im(p_B).$$
\end{lemma}
\begin{proof} It is immediate that $p_A$ and $p_B$ are projectors. The
  required isomorphism is given in the opposite direction by $p$.
\end{proof}

Given a smooth projective variety $X$ of dimension $d \geq 2$,
consider the decomposition constructed by Murre \cite{Murre}:
$$ \hspace{2cm} \h(X) =
\mathds{1} \oplus \h_1(X) \oplus M \oplus \h_{2d-1}(X) \oplus
\mathds{1}(d). \hspace{2cm} (*)$$ Since $\Hom(\mathds{1}(d),\h(X)) =
\Hom(\mathds{1}(d),\mathds{1}(d))$, we obtain
$$\Hom(\mathds{1}(d),\mathds{1}) = \Hom(\mathds{1}(d), \h_1(X)) =
\Hom(\mathds{1}(d),M) = \Hom(\mathds{1}(d), \h_{2d-1}(X)) =
0.$$
For any curve $C$, we have
$$\Hom(\h(C)(d-1),\h(X)) = \mathrm{Pic}(C \times X).$$
This implies that
$$\Hom(\h_{2d-1}(X),M) = 0.$$
By duality, we conclude that there are no morphisms going from the
right to the left in the decomposition of $\h(X)$ as above. By Lemma
\ref{triangular} applied several times, for any idempotent
endomorphism $p$ of $\h(X)$, we have a decomposition
$$\im(p) \simeq \im(p_0) \oplus \im(p_1) \oplus \im(p_M) \oplus
\im(p_{2d-1}) \oplus \im(p_{2d}),$$ where $\im(p)$ is a direct summand
in $\h(X)$ and where each direct summand appearing in the
decomposition of $\im(p)$ above is a direct summand of the
corresponding direct summand appearing in the decomposition $(*)$ of
$\h(X)$. The following proposition is then straightforward.

\begin{proposition} \label{albpic} Let $(X,p)$ be a motive. The
  idempotents $p_0$, $p_1$, $p_{2d-1}$ and $p_{2d}$ constructed above
  enjoy the following properties :

  $\bullet$ $(X,p_0)$ is isomorphic to $\mathds{1}^{\oplus n}$ for
  some $n$ and $H_*(X,p_0) = H_0(X,p)$.

  $\bullet$ $(X,p_1)$ is isomorphic to a direct summand of the
  $\h_1(C)$ for some curve $C$ and  $H_*(X,p_1) = H_1(X,p)$.

  $\bullet$ $(X,p_{2d-1})$ is isomorphic to a direct summand of the
  $\h_1(C)(d-1)$ for some curve $C$ and $H_*(X,p_{2d-1}) =
  H_{2d-1}(X,p)$.

  $\bullet$ $(X,p_{2d})$ is isomorphic to $\mathds{1}({d})^{\oplus n}$
  for some $n$ and $H_*(X,p_{2d}) = H_{2d}(X,p)$.
\end{proposition}

\begin{definition}
  The idempotent $p_1$ is called the \emph{Albanese projector} and the
  idempotent $p_{2d-1}$ is called the \emph{Picard projector}.
\end{definition}

\begin{remark}
  The Albanese and Picard projectors are not unique.
\end{remark}

As an immediate corollary, we can extend Murre's theorem on surfaces
\cite{Murre} to direct summands of Chow motives of surfaces.

\begin{theorem} \label{CKSurface} Let $M = (S,p)$ be a Chow motive
  where $S$ is a smooth projective surface. Then $M$ has a Murre
  decomposition. If, moreover, $M$ is Kimura finite-dimensional
  \cite{Kimura}, then $M$ satisfies Murre's conjecture (C).
\end{theorem}
\begin{proof}
  The correspondences $p_0$, $p_1$, $p_3$ and $p_4$ of Proposition
  \ref{albpic} together with $p_2 := p - \sum_{i \neq 2}p_i$ give a
  Chow--K\"unneth decomposition for $M$. That such a decomposition
  satisfies Murre's conjecture (B) is obvious.  Murre's conjecture (D)
  is possibly unclear only for one-cycles. Given a motive $N$, a
  correspondence $\gamma \in \Hom(\h(S),N)$ that acts trivially on
  $H^1(S)$ acts trivially on $\mathrm{Pic}^0_S = \CH^1(S)_\hom$. Since
  $p_2$ and $p_3$ are the only projectors that act possibly
  non-trivially on $\CH_1(S)$ and since $(p_3)_*\CH_1(M) \subseteq
  \CH_1(M)_\hom$, we get conjecture (D) for one-cycles on $M$, see also
  Proposition \ref{MurreC}. In the case when $M$ is Kimura
  finite-dimensional, Murre's conjecture (C) for $M$ can be obtained
  by applying \cite[Proposition 3.1]{Vial2}.
 \end{proof}

\section{Self-dual Chow--K\"unneth decompositions}
\label{CK}

Let $X$ be a smooth projective variety of dimension $d$ over $k$.  It
is proved in \cite[Theorem 4.2]{Vial3} that if the cohomology of $X$
in degree $\neq d$ is generated by the cohomology of curves, then $X$
admits a self-dual Chow--K\"unneth decomposition. Precisely, if
$H_i(X)=N^{\lfloor i/2 \rfloor}H_i(X)$ for all $i \neq d$, where $N$
is the coniveau filtration, then $X$ has a Chow--K\"unneth
decomposition.  It follows from Theorem \ref{isogeny}, together with a
decomposition of the diagonal argument \`a la Bloch--Srinivas, that a
fourfold which is fibred by rationally connected threefolds over a
curve has a self-dual Chow--K\"unneth decomposition \cite[Corollary
4.7]{Vial3}.  Del Angel and M\"uller-Stach \cite{dAMS} proved that
unirational threefolds have a Chow--K\"unneth decomposition. To do so,
they use Mori theory to reduce to the case of a conic fibration. In
this section, we generalise the aforementioned results by proving the
following:

\begin{theorem}\label{theoremCK}
  Let $f:X \r S$ be a dominant morphism defined over a field $k$ from
  a smooth projective variety $X$ to a smooth projective surface $S$
  such that the general fibre of $f_\Omega$ has trivial Chow group of
  zero-cycles. Suppose that $X$ has dimension $d \leq 4$. Then $X$
  has a self-dual Chow--K\"unneth decomposition $\{p_i\}_{0 \leq i
    \leq 2d}$.

  \noindent Moreover, this decomposition can be chosen so as to
  satisfy the following properties: \medskip

  $\bullet$ $p_{0}$ factors through a point $P_0$, i.e.  $(X,p_{0})$
  is isomorphic to $\h(P_0)$.

  $\bullet$ $p_{1}$ and $p_3$ factor through a curve, i.e.  there
  is a curve $C_0$ (resp. $C_1$) such that $(X,p_{1})$ (resp.
  $(X,p_3)$) is a direct summand of $\h_1(C_0)$ (resp.
  $\h_1(C_1)(1)$).

  $\bullet$ $p_{2}$ factors through a surface, i.e.  there is a
  surface $S_0$ such that $(X,p_{2})$ is isomorphic to a direct
  summand of $\h(S_0)$.

  $\bullet$ If $d\leq 4$, $p_4$ factors through a surface, i.e.  there
  is a surface $S_1$ such that $(X,p_{4})$ is isomorphic to a direct
  summand of $\h(S_1)(1)$.
\end{theorem}

\noindent In particular, this will give an alternate proof to del
Angel and M\"uller-Stach's result for conic fibrations over a surface.

\noindent We divide the proof into several steps.

\subsection{The projectors $\pi_0$, $\pi_1$ and $\pi_2^{tr}$}
\label{firstproj}

The surface $S$ has a Chow--K\"unneth decomposition \cite{Murre,
  Scholl} $\{\pi_0^S, \pi_1^S, \pi_2^S, {}^t\pi_1^S, {}^t\pi_0^S\}$.
The motive $(S, \pi_2^S)$ admits a direct summand $(S, \pi_2^{tr,S})$
called its transcendental part, cf \cite{KMP}. The action of the
idempotent $\pi_2^{tr,S}$ on the homology of $S$ is the orthogonal
projector on the orthogonal complement for cup-product of the span of
the classes of algebraic one-cycles inside $H_2(S)$. In characteristic
zero, for Betti cohomology, $(\pi_2^{tr,S})_*H_*(S)$ is the sub-Hodge
structure of $H_2(S)$ generated by $H^{2,0}(S)=H^2(S,O_S)$ thanks to
the Lefschetz $(1,1)$-theorem. The idempotent $\pi_2^{tr,S}$ acts
trivially on $\CH_1(S_\Omega)$ and on $\CH_2(S_\Omega)$ so that
$\CH_*(S,\pi_2^{tr,S}) = \CH_0(S,\pi_2^{tr,S})$.  \medskip

By Proposition \ref{surjective}, there is a correspondence $\Gamma \in
\CH^{d}(S \times X)$ such that $\Gamma_f \circ \Gamma = \Delta_S$, so
that the Chow motive of $S$ is a direct summand of the Chow motive of
$X$. We thus define mutually orthogonal idempotents $\pi_0 := \Gamma
\circ \pi_0^S \circ \Gamma_f$, $\pi_1 := \Gamma \circ \pi_1^S\circ
\Gamma_f$ and $\pi_2^{tr} := \Gamma \circ \pi_2^{tr, S} \circ
\Gamma_f$. Because the idempotents $\pi_0^S$, $\pi_1^S$ and
$\pi_2^{tr, S}$ in $\CH_2(S \times S)$ satisfy $(\pi_0^S)_*H_*(S) =
H_0(S)$, $(\pi_1^S)_*H_*(S) = H_1(S)$ and $(\pi_2^{tr,S})_*H_*(S)
\subset H_2(S)$, we see that $(\pi_0)_*H_*(X) \subset H_0(X)$,
$(\pi_1)_*H_*(X) \subset H_1(X)$ and $(\pi_2^{tr})_*H_*(X) \subset
H_2(X)$.\medskip

Since $\CH_0(S_\Omega) = (\pi_0^S + \pi_1^S + \pi_2^{tr,
  S})_*\CH_0(S_\Omega)$ and since both $(\Gamma_f)_* : \CH_0(X_\Omega)
\r \CH_0(S_\Omega)$ and $\Gamma_* : \CH_0(S_\Omega) \r \CH_0(X_\Omega)$
are isomorphisms by Theorem \ref{isogeny}, we get

\begin{proposition} \label{vanishingchowS} $(\pi_0 + \pi_1 +
  \pi_2^{tr})_*\CH_0(X_\Omega) = \CH_0(X_\Omega)$. \qed
\end{proposition}

This yields that the decomposition $\h(X) = (X,\pi_0) \oplus (X,\pi_1)
\oplus (X,\pi_2^{tr}) \oplus M$ satisfies $\CH_0(M_\Omega)=0$. Theorem
\ref{effective} gives a smooth projective variety $Y$ of dimension one
less than the dimension $d$ of $X$ together with an idempotent $q \in
\CH_{d-1}(Y \times Y)$ such that $M \simeq (Y,q,1)$.\medskip

By definition $H_*(Y,q,1) = H_{*-2}(Y,q) = q_*H_{*-2}(Y)$.
Consequently, we see that $H_0(M)=0$ and also that $H_1(M)=0$.
Therefore, $(\pi_0)_*H_*(X) = H_0(X)$ and $(\pi_1)_*H_*(X)= H_1(X)$.
\medskip

It is interesting to note that we can show that the $\pi_i$'s act as
the K\"unneth projectors on homology only after having determined
their action on Chow groups.

\subsection{Chow--K\"unneth decomposition for $\dim X=3$}
\label{CK3}

When $\dim X=3$, the Chow motive $\h(X)$ decomposes as $(X,\pi_0)
\oplus (X,\pi_1) \oplus (X,\pi_2^{tr}) \oplus M$ where $M$ is
isomorphic to a motive $(Y,q,1)$ with $\dim Y = 2$. In other words,
$\h(X)$ is isomorphic to a direct sum of direct summands of twisted
motives of surfaces. Theorem \ref{CKSurface} then says that $X$ has a
Murre decomposition.  This is made more precise in \S\ref{dim3}.

\subsection{A first approach to splitting the motive of $X$ when $\dim
  X \geq 3$} \label{dimsup3}

Ultimately, our goal is to define a self-dual Chow--K\"unneth
decomposition for $X$ with $\dim X \leq 4$. Let's thus study the
orthogonality relations between the idempotents $\pi_0$, $\pi_1$,
$\pi^{tr}_2$ and their transpose ${}^t\pi_0$, ${}^t\pi_1$,
${}^t\pi^{tr}_2$. We already know that $\pi_0$, $\pi_1$ and
$\pi^{tr}_2$ are mutually orthogonal. For dimension reasons (see also
\S\ref{orthoproj}), the only possible missing orthogonality relations
concern $\pi^{tr}_2 \circ {}^t\pi^{tr}_2$ and ${}^t\pi^{tr}_2 \circ
\pi^{tr}_2$. The crucial point here is that $\pi^{tr}_2 \circ
{}^t\pi^{tr}_2 = 0$ and that ${}^t\pi^{tr}_2 \circ \pi^{tr}_2$ acts
trivially on $\CH_*(X_\Omega)$. Let's prove these facts.\medskip

Proposition \ref{pi2vanish} gives two proofs that $\pi^{tr}_2 \circ
{}^t\pi^{tr}_2 = 0$. The first proof relies on Lemma \ref{1way} and is
particular to our geometric situation. The second proof relies on
Lemma \ref{2way}; it is more general and could be useful in other
situations.\medskip

On the one hand, we have

\begin{lemma} \label{1way} Let $f : X \r S$ be a dominant map between
  two smooth projective varieties with $\dim X > \dim S$. Then
  $\Gamma_f \circ {}^t\Gamma_f = 0$.
\end{lemma}
\begin{proof}
  By definition, we have $\Gamma_f \circ {}^t\Gamma_f =
  (p_{1,3})_*(p_{1,2}^*{}^t\Gamma_f \cap p_{2,3}^*\Gamma_f)$, where
  $p_{i,j}$ denotes projection from $S \times X \times S$ to the
  $(i,j)$-th factor. These projections are flat morphisms, therefore,
  by flat pullback, we have $p_{1,2}^*{}^t\Gamma_f = [{}^t\Gamma_f
  \times S]$ and $p_{2,3}^*\Gamma_f = [S \times \Gamma_f]$. It is easy
  to see that the closed subschemes ${}^t\Gamma_f \times S$ and $S
  \times \Gamma_f$ of $S \times X \times S$ intersect properly. Their
  intersection is given by $\{(f(x),x,f(x)) : x \in X\} \subset S
  \times X \times S$. This is a closed subset of dimension $\dim X$
  and its image under the projection $p_{1,3}$ has dimension $\dim S$,
  which is strictly less than $\dim X$ by assumption. The projection
  $p_{1,3}$ is a proper map and hence, by proper pushforward, we get
  that $(p_{1,3})_* [\{(f(x),x,f(x)) \in S \times X \times S : x \in X
  \}] =0$.
\end{proof}

On the other hand, we have the following two lemmas, the first of
which will be used later on.

  \begin{lemma} \label{trivialaction} Let $\gamma \in \CH^0(V \times
     W)$ be a correspondence such that $\gamma_*$ acts trivially on
     zero-cycles. Then $\gamma = 0$.
\end{lemma}
\begin{proof} We can assume that $V$ and $W$ are both connected.  The
  cycle $\gamma$ is equal to $a\cdot [V \times W]$ for some $a \in
  \Q$. Let $z$ be a zero-cycle on $V$. Then $\gamma_*z = a \cdot \deg
  z \cdot [W]$. This immediately implies $a=0$.
\end{proof}

\begin{lemma} \label{2way} Let $\gamma \in \CH^1(V \times W)$ be a
  correspondence such that both $\gamma_*$ and $\gamma^*$ act
  trivially on zero-cycles after base-change to an algebraically
  closed field over $k$. Then $\gamma = 0$.
\end{lemma}
\begin{proof} Since base-change to a field extension induces an
  injective map on Chow groups with rational coefficients, we may
  assume that the base field $k$ is algebraically closed. We may also
  assume that $V$ and $W$ are connected. We have
  $$\mathrm{Pic}(V \times W) = \mathrm{Pic}(V) \times [W] \oplus [V]
  \times \mathrm{Pic}( W) \oplus
  \mathrm{Hom}(\mathrm{Alb}_V,\mathrm{Pic}^0_W)\otimes \Q.$$ Let
  $\varphi \in \mathrm{Hom}(\mathrm{Alb}_V,\mathrm{Pic}^0_W)\otimes
  \Q$ be the component of $\gamma$ under the above decomposition. By
  assumption, $\gamma_*$ acts trivially on $\CH_0(V)$ and, hence, also
  on $\CH_0(V)_\hom$. The albanese map $\CH_0(V)_\hom \r
  \mathrm{Alb}_V(k)$ is surjective and it follows that $\varphi=0$.
  The cycle $\gamma$ is thus equal to $D_1 \times [W] + [V] \times
  D_2$ for some divisors $D_1 \in \CH^1(V)$ and $D_2 \in \CH^1(W)$. Let
  $z$ be a zero-cycle on $V$.  Then $\gamma_*z = \deg z \cdot D_2$.
  This immediately implies that $D_2=0$. Likewise, if $z \in \CH_0(W)$,
  then $\gamma^* z =0$ implies $D_1=0$. We have thus proved that
  $\gamma=0$.
\end{proof}

\begin{proposition} \label{pi2vanish} $\pi_2^{tr} \circ {}^t\pi_2^{tr}
  =0$. More generally, $\Hom(
  (X,{}^t\pi_2^{tr}),(X,\pi_2^{tr}))=0$.
 \end{proposition}
\begin{proof}
  From Lemma \ref{1way} and from the very definition of $\pi_2^{tr}$,
  it is immediate that $\pi_2^{tr} \circ {}^t\pi_2^{tr} =0$.  Let now
  $\alpha$ be a correspondence in $\CH_d(X \times X)$. The
  correspondence $\pi_2 \circ \alpha \circ {}^t\pi_{2}^{tr}$ factors
  through a correspondence $\gamma \circ \pi_2^{tr,S} \in \CH_d(S
  \times S)$. If $d >4$, then the statement is clear. If $d=4$, we use
  the fact that $\pi_2^{tr,S}$ sends zero-cycles on $S$ to zero-cycles
  in the Albanese kernel of $S$. Hence, $\gamma\circ \pi_2^{tr,S}$
  sends zero-cycles on $S$ to homologically trivial two-cycles on $S$.
  In particular, $\gamma\circ \pi_2^{tr,S}$ acts trivially on
  zero-cycles on $S$ and we can therefore apply Lemma
  \ref{trivialaction}. Let's now assume that $d=3$ and let's give a
  proof using Lemma \ref{2way} when the base field is a subfield of
  $\C$.  The reason is that we use Abel--Jacobi maps (although it is
  almost certainly true that the Albanese variety and the Picard
  variety enjoy the required functoriality properties over any base
  field).  The correspondence $\pi_2^{tr} \circ \alpha \circ
  {}^t\pi_2^{tr} \in \CH_3(X \times X)$ factors through a
  correspondence $\pi_2^{tr, S} \circ \gamma \circ {}^t\pi_2^{tr,S}
  \in \CH^1(S \times S)$. In particular, by functoriality of the
  Abel--Jacobi map, $\pi_2^{tr, S} \circ \gamma \circ
  {}^t\pi_2^{tr,S}$ sends $0$-cycles on $S_\C$ to $1$-cycles on $S_\C$
  which lie in the kernel of the Abel--Jacobi map.  This last kernel
  is trivial.  Therefore, $\pi_2^{tr, S} \circ \gamma \circ
  {}^t\pi_2^{tr,S}$ acts trivially on zero-cycles on $S_\Omega$.
  Clearly the same holds for its transpose.  Therefore, $\pi_2^{tr, S}
  \circ \gamma \circ {}^t\pi_2^{tr,S} = 0$ and, hence, $\pi_2^{tr} \circ
  \alpha \circ {}^t\pi_2^{tr} =0$.
\end{proof}

\begin{proposition} \label{trivialactionp2}
  ${}^t\pi_2^{tr} \circ \pi_2^{tr} $ acts trivially on $\CH_*(X_\Omega)$.
\end{proposition}
\begin{proof}
  The correspondence ${}^t\pi_2^{tr} \circ \pi_2^{tr} $ factors
  through a correspondence $\gamma\circ \pi_2^{tr,S} \in \CH_{4-d}(S
  \times S)$. The proposition follows immediately since the idempotent
  $ \pi_2^{tr,S}$ acts trivially on $\CH_1(X_\Omega)$ and on
  $\CH_2(X_\Omega)$.
\end{proof}

Let's then define $p_0:=\pi_0$, $p_1:=\pi_1$ and
$p_2^{tr}:=(1-\frac{1}{2}{}^t\pi_2^{tr}) \circ \pi_2^{tr}$. It is
clear that these are idempotents and that $\{p_0, p_1, p_2^{tr},
{}^tp_2^{tr}, {}^tp_1, {}^tp_0 \}$ is a set of mutually orthogonal
idempotents in $\CH_d(X \times X)$. This yields a splitting $$\h(X) =
(X,p_0) \oplus (X, p_1) \oplus (X, p_2^{tr}) \oplus (X, {}^tp_2^{tr})
\oplus (X, {}^tp_1) \oplus (X, {}^tp_0) \oplus M.$$

\begin{proposition} \label{vanishingchow} $(p_0 + p_1 +
  p_2^{tr})_*\CH_0(X_\Omega) = \CH_0(X_\Omega)$.
\end{proposition}
\begin{proof} By Proposition \ref{trivialactionp2} we see that
  $(p_2^{tr})_*x = (\pi_2^{tr})_*x$ for all $X \in \CH_*(X_\Omega)$. We
  can therefore conclude with Proposition \ref{vanishingchowS}.
 \end{proof}

 \begin{theorem} \label{finedec} There exists a smooth projective
   variety $Z$ of dimension $d-2$ and an idempotent $q \in \CH_{d-2}(Z
   \times Z)$ such that $$\h(X) = (X,p_0) \oplus (X, p_1) \oplus (X,
   p_2^{tr}) \oplus (X, {}^tp_2^{tr}) \oplus (X, {}^tp_1) \oplus (X,
   {}^tp_0) \oplus (Z,q,1).$$
\end{theorem}
\begin{proof} The theorem is a combination of Proposition
  \ref{vanishingchow} and Corollary \ref{effective-coro}.
 \end{proof}

\begin{theorem} If $d \leq 4$, then $X$ has a Murre decomposition.
\end{theorem}
\begin{proof} The theorem follows from Theorem \ref{finedec} and
  Theorem \ref{CKSurface}.
 \end{proof}

 Let's write $\h(X) = (X,p) \oplus (Z,q,1)$, where $p = p_0+ p_1+
 p_2^{tr}+{}^tp_2^{tr}+ {}^tp_1+ {}^tp_0$.  Although $\h(X) =
 \h(X)^\vee(d) = (X,p) \oplus (Z,{}^tq,1)$, it is not clear that
 $(Z,q,1)$ is self-dual, i.e. isomorphic to $(Z,{}^tq,1)$. Thus we
 need to refine the above construction.

\subsection{The projectors $\pi_2^{alg}$ and $\pi_3$} \label{proj3}

Until \S \ref{dim3}, the dimension $d$ of $X$ is supposed to be $\geq
4$. \medskip

Let's go back to the situation and notations of \S \ref{firstproj}.
Let $p := \Delta_X - (\pi_0 + \pi_1 + \pi_2^{tr})$. We have the
decomposition $\h(X) = (X,\pi_0) \oplus (X,\pi_1) \oplus
(X,\pi_2^{tr}) \oplus M$ with $M = (X,p)$ isomorphic to $(Y,q,1)$.
Choose an isomorphism $f : (Y,q,1) \r M$ and let $g : M \r (Y,q,1)$ be
its inverse. Let $q_0^Y$ and $q_1^Y$ be respectively the point
projector and the Albanese projector of Proposition \ref{albpic} for
$(Y,q,0)$.  We define idempotents $\pi_2^{alg} := f \circ q_0^Y \circ
g$ and $\pi_3 : = f \circ q_1^Y \circ g$.  \medskip

These two idempotents are orthogonal and are obviously orthogonal to
the idempotents $\pi_0$, $\pi_1$ and $\pi_2^{tr}$ previously defined.
Their action on cohomology is the expected one: we have $H_2(X) =
H_2(X,\pi_2^{tr}) \oplus H_2(M)$ but $H_2(M) = H_0(Y,q) =
H_0(Y,\pi_0^Y)$. Therefore $\pi_2 := \pi_2^{tr} + \pi_2^{alg}$ induces
the K\"unneth projector $H_*(X) \r H_2(X) \r H_*(X)$. We also have
$H_3(X) = H_3(M) = H_1(Y,q) = H_1(Y,\pi_1^Y)$ and hence
$(\pi_3)_*H_*(X) = H_3(X)$.

\subsection{The remaining projectors} \label{remainingproj}

We now define $\pi_{2d} := {}^t\pi_0$, $\pi_{2d-1} := {}^t\pi_1$,
$\pi_{2d-2} := {}^t\pi_2$ and $\pi_{2d-3} := {}^t\pi_3$. By Poincar\'e
duality, these idempotents satisfy $(\pi_i)_*H_*(X) = H_i(X)$.

\subsection{Orthonormalising the projectors} \label{orthoproj}

We have the following non-commutative Gram--Schmidt process \cite[Lemma
2.12]{Vial3}

\begin{lemma} \label{linalg} Let $V$ be a $\Q$-algebra and let $k$ and
  $n$ be positive integers. Let $\pi_0, \ldots, \pi_n$ be idempotents
  in $V$ such that $\pi_i \circ \pi_j = 0$ whenever $i -j < k$ and $i
  \neq j$. Then the endomorphisms $$p_i := (1-\frac{1}{2}\pi_n) \circ
  \cdots \circ (1-\frac{1}{2}\pi_{i+1}) \circ \pi_i \circ
  (1-\frac{1}{2}\pi_{i-1}) \circ \cdots \circ (1-\frac{1}{2}\pi_0)$$
  define idempotents such that $p_i \circ p_j = 0$ whenever $i -j <
  k+1$ and $i \neq j$.
\end{lemma}

Let's state an orthonormalisation result which is valid in a general
setting and that we can apply to our particular case of interest.

\begin{theorem} \label{GS} Let $X$ be a smooth projective variety of
  dimension $d$.  Let $\pi_0, \ldots, \pi_n \in \CH_d(X \times X)$ be
  idempotents such that $\pi_r \circ \pi_s =0$ for all $ r < s $. Then
  applying $n$ times the Gram--Schmidt process of Lemma \ref{linalg}
  gives mutually orthogonal idempotents $p_0, \ldots, p_n$ such that
  $(X,p_r) \simeq (X,\pi_r)$ for all $r$. Furthermore, \medskip

$\bullet$ if there exists $r$ such that
  $(\pi_r)_*H_*(X) = H_r(X)$, then $(p_r)_*H_*(X) =
  H_r(X)$.

$\bullet$ if $\pi_r = {}^t\pi_{n-r}$ for all $r$, then $p_r =
{}^tp_{n-r}$ for all $r$;

$\bullet$ if there exists $r$ such that $\pi_s \circ \pi_r$ acts
trivially on $\CH_*(X)$ for all $s > r$, then $(p_r)_*\CH_*(X) =
(\pi_r)_*\CH_*(X)$ inside $\CH_*(X)$.
\end{theorem}
\begin{proof} The idempotents $\pi_0, \ldots, \pi_n$ satisfy the
  assumptions of Lemma \ref{linalg} with $k=1$. Therefore, after
  having run $n$ times the orthonormalisation process of Lemma
  \ref{linalg}, we get mutually orthogonal idempotents. In order to
  prove the theorem, it suffices to prove each statement after each
  application of the orthonormalisation process. Given $r$, the
  isomorphism $(X,p_r) \simeq (X,\pi_r)$ is simply given by the
  correspondence $\pi_r \circ p_r$ ; its inverse is $p_r \circ \pi_r$
  as can be readily checked.

  If $(\pi_r)_*H_*(X) = H_r(X)$, then the image of $\pi_r$ in $H_d(X
  \times X) \simeq \Hom(H_*(X),H_*(X))$ is central. Therefore, if
  $\pi_r \circ \pi_s =0$ for all $ r < s $, then the image of $\pi_s
  \circ \pi_r$ in $H_d(X \times X)$ is trivial for all $s > r$. It is
  then straightforward to conclude that $(p_r)_*H_*(X) = H_r(X)$.

  If $\pi_r = {}^t\pi_{n-r}$ for all $r$, then it is straightforward
  to check from the formula of Lemma \ref{linalg} that $p_r =
  {}^tp_{n-r}$ for all $r$

  Let's fix $r$. Given the isomorphism $(X,p_r) \simeq (X,\pi_r)$, it
  is very tempting to conclude that $(p_r)_*\CH_*(X_\Omega) =
  (\pi_r)_*\CH_*(X_\Omega)$ in $\CH_*(X_\Omega)$. However, this appears
  not to be obvious at all and a careful analysis of the
  non-commutative Gram--Schmidt process needs to be carried on.  By
  examining the formula defining the idempotent $p_r$, together with
  the assumption that $\pi_s \circ \pi_r$ acts trivially on $\CH_*(X)$
  for all $s > r$, we see that, for $x \in \CH_*(X_\Omega)$, we have
  $(p_r)_*x = (\pi_r)_*x \in \CH_*(X_\Omega)$.
\end{proof}

We wish to apply Theorem \ref{GS} to the set of idempotents
$\{\pi_0,\pi_1,\pi_2,\pi_3, \pi_{2d-3}, \pi_{2d-2}, \pi_{2d-1},
\pi_{2d}\}$. In order to do so, we have to show that $\pi_i \circ
\pi_j = 0$ whenever $i < j$. We already know that $\pi_0$, $\pi_1$,
$\pi_2$ and $\pi_3$ are mutually orthogonal. Let's prove the missing
orthogonality relations. First we have :
\begin{itemize}
\item $\pi_0 \circ {}^t\pi_{0} = \pi_0 \circ {}^t\pi_{1} = \pi_0 \circ
  {}^t\pi_{2} = \pi_0 \circ {}^t\pi_{3} = 0$.
\item $\pi_1 \circ {}^t\pi_{1} = \pi_1 \circ
  {}^t\pi_{2} = \pi_1 \circ {}^t\pi_{3} = 0$.
\item $\pi_2 \circ {}^t\pi_{2}^{alg} = 0$ and hence $\pi_2 \circ
  {}^t\pi_2 = 0$ thanks to Proposition \ref{pi2vanish}.
  \end{itemize}

  These relations are obvious : one uses a dimension argument as well
  as the fact that $\pi_0$ (resp. $\pi_1$, $\pi_2^{tr}$,
  $\pi_2^{alg}$, $\pi_3$) factors through a variety $P_0$ (resp.
  $C_0$, $S$, $P_1$, $C_1$) of dimension $0$ (resp. $1$, $2$, $0$,
  $1$).  For instance, $\pi_1 \circ {}^t\pi_{3}$ factors through a
  correspondence in $\CH_{d-1}(C_1 \times C_0)$. If $d \geq 4$, then
  this last group is trivial.\medskip

  Using the same arguments, the following orthogonality relations can
  be further proved. These relations are not necessary to run the
  non-commutative Gram--Schmidt process but are essential to the proof
  of Proposition \ref{vanishingchow2}.
  \begin{itemize}
  \item $ {}^t\pi_0 \circ \pi_{0} = {}^t\pi_0 \circ \pi_{1} =
    {}^t\pi_0 \circ \pi_{2} = {}^t\pi_0 \circ \pi_{3} = 0$.
  \item $ {}^t\pi_1 \circ \pi_{1} = {}^t\pi_1 \circ \pi_{2} =
    {}^t\pi_1 \circ \pi_{3} = 0$.
  \item $ {}^t\pi_2 \circ \pi_{2}^{alg} = 0$.
  \item $ {}^t\pi_3 \circ \pi_{2}^{alg} = 0$.
  \end{itemize}

  Secondly, the remaining orthogonality relations needed to run the
  non-commutative Gram--Schmidt process follow from Lemma
  \ref{trivialaction}.

 \begin{itemize}
 \item $\pi_2 \circ {}^t\pi_{3} = 0$. The correspondence $\pi_2 \circ
   {}^t\pi_{3}$ factors through a correspondence $\gamma \in
   \CH_{d-1}(C_1 \times S_0)$, where $S_0$ is a surface, that sends
   zero-cycles to homologically trivial cycles on $S_0$. Again, if
   $d>4$, then the result is trivial. If $d=4$, then we conclude by
   Lemma \ref{trivialaction}.
 \item $\pi_3 \circ {}^t\pi_{3} = 0$. The correspondence $\pi_3 \circ
   {}^t\pi_{3}$ factors through a correspondence $\gamma \in
   \CH_{d-2}(C_1 \times C_1)$ that sends zero-cycles to homologically
   trivial cycles on $C$.  Again, if $d>4$, then the result is
   trivial. If $d=4$, then we conclude by Lemma \ref{trivialaction}.
\end{itemize}

We are now in a position to apply Theorem \ref{GS} to obtain a set of
mutually orthogonal idempotents $\{p_0,p_1,p_2,p_3, p_{2d-3},
p_{2d-2}, p_{2d-1}, p_{2d}\}$ such that $p_{2d-i}={}^tp_i$ which
induce the expected K\"unneth projectors modulo homological
equivalence.

\begin{remark} \label{missingorth} It follows from the above
  discussion that the only possible missing orthogonality relations
  among the idempotents $\pi_0,\pi_1,\pi_2^{alg},\pi_2^{tr}, \pi_3$
  and their transpose are the following.
  \begin{itemize}
  \item ${}^t\pi_3 \circ \pi_2^{tr}$.
  \item ${}^t\pi_2^{tr} \circ \pi_2^{tr}$.
  \item ${}^t\pi_3 \circ \pi_3$.
  \end{itemize}
  It can then be checked that it is actually enough to run the
  non-commutative Gram--Schmidt process only once on the set of
  idempotents $\{\pi_0,\pi_1,\pi_2,\pi_3, \pi_{2d-3}, \pi_{2d-2},
  \pi_{2d-1}, \pi_{2d}\}$ to obtain a set of mutually orthogonal
  idempotents. We can therefore describe the $p_i$'s in terms of the
  $\pi_i$'s by not too complicated explicit formulas. Such formulas
  may then be used, for instance, to give a quicker proof of the motivic
  Lefschetz conjecture for $X$. However, we describe a method that
  might be useful in other situations where the Gram--Schmidt process
  needs to be run several times.
\end{remark}

The following proposition is fundamental to proving Proposition
\ref{vanishingchow2} and, hence, to proving Murre's conjectures for
$X$.

\begin{proposition} \label{trivialactionchow} Let $p$ and $q$ be any
  two distinct idempotents among the idempotents $\pi_0$, $\pi_1$,
  $\pi_2$, $\pi_3$, $\pi_{2d-3}$, $\pi_{2d-2}$, $\pi_{2d-1}$ and
  $\pi_{2d}$.  Then $p \circ q$ acts trivially on $\CH_l(X_\Omega)$ for
  all $l$.
\end{proposition}
\begin{proof}
  From remark \ref{missingorth} we only need to prove that ${}^t\pi_3
  \circ \pi_2^{tr}$, ${}^t\pi_2^{tr} \circ \pi_2^{tr}$ and ${}^t\pi_3
  \circ \pi_3$ act trivially on $\CH_*(X_\Omega)$. In the first case,
  ${}^t\pi_3 \circ \pi_2^{tr}$ factors through a correspondence
  $\gamma \circ \pi_2^{tr,S} \in \CH_0(S \times C_1)$ for some curve
  $C_1$ and it therefore acts trivially on $\CH_*(X_\Omega)$ because $
  \pi_2^{tr,S}$ only acts possibly non-trivially on $\CH_0(S_\Omega)$.
  In the second case, ${}^t\pi_2^{tr} \circ \pi_2^{tr}$ factors
  through a correspondence $\gamma \circ \pi_2^{tr,S} \in \CH_0(S
  \times S)$ and we conclude in the same way. In the last case,
  ${}^t\pi_3 \circ \pi_3$ factors through a correspondence $\gamma
  \circ \pi_1^{C_1} \in \CH_0(C_1 \times C_1)$ which also acts
  trivially on $\CH_*((C_1)_\Omega)$ because $\pi_1^{C_1}$ acts
  trivially on $\CH_1((C_1)_\Omega)$.
\end{proof}

\begin{proposition} \label{vanishingchow2}
  $(p_0 + p_1 + p_2)_*\CH_0(X_\Omega) = \CH_0(X_\Omega)$.
\end{proposition}
\begin{proof}
  Proposition \ref{trivialactionchow} and Theorem \ref{GS} imply that
  $(\pi_i)_*x = (p_i)_*x$ for $i = 0, 1$ or $2$ and for all $x \in
  \CH_0(X_\Omega)$. By Proposition \ref{vanishingchowS}, this yields
  $(p_0+p_1 + p_2)_*\CH_0(X_\Omega) = \CH_0(X_\Omega)$ as claimed.
\end{proof}

Finally, when $d=4$, we define $p_4:=\Delta_X - \sum_{i \neq 4} p_i$.
The set $\{p_i\}_{0 \leq i \leq 8}$ is then a self-dual Chow--K\"unneth
decomposition for $X$.  Moreover, $p_4$ has the following property.

\begin{proposition}
  $(X,p_4)$ is isomorphic to a direct summand of $\h(S_1)(1)$ for some
  smooth projective surface $S_1$.
\end{proposition}
\begin{proof}
  Proposition \ref{vanishingchow2} implies that
  $(p_4)_*\CH_0(X_\Omega)=0$. Also we know that $p_4 = {}^tp_4$.
  Therefore, the result follows immediately from Corollary
  \ref{effective-coro}.
\end{proof}

\subsection{Back to the case $\dim X = 3$} \label{dim3}

Let's now consider the case of a conic fibration over a surface. In
section \ref{CK3}, we already gave a quick argument showing that $X$
has a Chow--K\"unneth decomposition. As in the case $\dim X =4$, we
want to show that a Chow--K\"unneth decomposition for $X$ can be chosen
to be self-dual, a result which is not shown in \cite{dAMS}. In order
to prove Murre's conjectures for such a decomposition (which will be
done in section \ref{Murreconj}), we also want to show that the middle
idempotent factors through a curve.  \medskip

For this purpose, we define $\pi_0$, $\pi_1$, $\pi_2^{tr}$,
$\pi_2^{alg}$ and $\pi_2 := \pi_2^{tr} + \pi_2^{alg}$ the same way we
did in sections \ref{firstproj} and \ref{proj3}. The only difference
with section \ref{proj3} is that we don't define an idempotent
$\pi_3$. We then define $\pi_6 = {}^t\pi_0$, $\pi_5 = {}^t\pi_1$ and
$\pi_4 = {}^t\pi_2$.  As in sections \ref{firstproj} and \ref{proj3},
it is easy to see that these do define the K\"unneth projectors in
homology.

As before, we have $\pi_i \circ \pi_j = 0$ for all $i<j$ not equal to
$3$. These relations make it possible to run the non-commutative
Gram--Schmidt process and to get mutually orthogonal
idempotents $p_0, p_1, p_2, p_4, p_5, p_6$ such that
$(p_i)_*H_*(X)=H_i(X)$ and $p_{6-i} = {}^tp_i$ for all $i \neq 3$.
Setting $p_3 := \Delta_X - \sum_{i\neq 3} p_i$, we thus get a
self-dual Chow--K\"unneth decomposition for $X$.  Again, as before, we
have that $\pi_j \circ \pi_i$ acts trivially on $\CH_*(X_\Omega)$ for
all $j>i$ not equal to $3$. The middle idempotent $p_3$ has thus the
following property.

\begin{proposition} \label{middle-curve}
   There exists a curve $C_1$ such that $(X,p_3)$ is isomorphic to a
  direct summand of $\h_1(C_1)(1)$.
\end{proposition}
\begin{proof} As in the proof of Proposition \ref{vanishingchow2}, we
  have $(p_0 + p_1 + p_2)_*\CH_0(X_\Omega) = \CH_0(X_\Omega)$. Therefore,
  $\CH_0(X_\Omega,p_3)=0$.  But $p_3 = {}^tp_3$, so that
  $\CH_0(X_\Omega,{}^tp_3)=0$ too. It follows from Corollary
  \ref{effective-coro} that there exists a curve $C_1$ such that
  $(X,p_3)$ is isomorphic to a direct summand of $\h(C_1)(1)$. The
  fact that the motive $(X,p_3)$ is pure of weight $3$ makes it
  possible to conclude.
 \end{proof}

\section{The motivic Lefschetz conjecture for $X$} \label{motLef}

Let $X$ be a smooth projective variety of dimension $d$ over a field
$k$. Let $i \leq d$ and let $\iota : H \r X$ be a smooth linear
section of dimension $i$ and let $L := (\iota,\id_X)_*\Gamma_\iota =
\Gamma_{\iota} \circ {}^t\Gamma_\iota \in \CH_{i}(X \times X)$. The
correspondence $L$ acts on cohomology or Chow groups as intersecting
$d-i$ times by a smooth hyperplane section of $X$. The variety $X$ is
said to satisfy the motivic Lefschetz conjecture in degree $i$ if
there exist mutually orthogonal idempotents $\pi_i$ and $\pi_{2d-i}$
such that $H_*(X,\pi_i)=H_i(X)$ and $H_*(X,\pi_{2d-i})=H_{2d-i}(X)$
and such that the induced map $$L : (X,\pi_{2d-i}) \r (X,\pi_i,d-i)$$
is an isomorphism of Chow motives.  The variety $X$ is said to satisfy
the \emph{motivic Lefschetz conjecture} if it satisfies the motivic
Lefschetz conjecture in all degrees $<d$.  Note that if $X$ satisfies
the motivic Lefschetz conjecture in degree $i$ then $X$ satisfies the
Lefschetz standard conjecture in degree $i$, i.e. there exists a
correspondence $\Gamma \in \CH^{i}(X \times X)$ such that $\Gamma_* :
H_i(X) \r H_{2d-i}(X)$ is the inverse to $L : H_{2d-i}(X) \r H_i(X)$.
The motivic Lefschetz conjecture for $X$ follows from a combination of
the Lefschetz standard conjecture for $X$ and of Kimura's finite
dimensionality conjecture for $X$ ; it is thus expected to hold for
all smooth projective varieties.

\begin{proposition} \label{point} Let $P$ be a zero-dimensional
  variety over $k$.  Let $p \in \CH_d(X \times X)$ be an idempotent
  such that $(X,p)$ is isomorphic to $\h(P)(i)$ for some integer $i$
  satisfying $2i \leq d$. If the induced map $L : H_{2d-2i}(X,{}^tp)
  \r H_{2i}(X,p)$ is an isomorphism, then $L : (X,{}^tp) \r
  (X,p,d-2i)$ is an isomorphism of Chow motives.
\end{proposition}
\begin{proof}
  There exist correspondences $f \in \Hom(\h(P)(i),(X,p))$ and $g \in
  \Hom((X,p),\h(P)(i))$ such that $g\circ f = \id_{\h(P)(i)}$ and $f
  \circ g = p$. The correspondence $g \circ L \circ {}^tg \in
  \End(\h(P))$ induces an automorphism of $H_{0}(P)$ and, hence, is
  itself an automorphism. Therefore, it admits an inverse $\alpha \in
  \End(\h(P))$. It is now straightforward to check that ${}^tp
  \circ {}^tg \circ \alpha \circ g \circ p$ is the inverse of $p \circ
  L \circ {}^tp$.
\end{proof}

\begin{proposition} \label{curve} Let $J$ be an abelian variety over
  $k$.  Let $p \in \CH_d(X \times X)$ be an idempotent such that
  $(X,p)$ is isomorphic to $\h_1(J)(i)$ for some integer $i$
  satisfying $2i+1 \leq d$ and such that $p$ is orthogonal to ${}^tp$
  (this last condition is automatically satisfied if $2i+1 <d-1$). If
  the induced map $L : H_{2d-2i-1}(X,{}^tp) \r H_{2i+1}(X,p)$ is an
  isomorphism, then $L : (X,{}^tp) \r (X,p,d-2i-1)$ is an isomorphism
  of Chow motives.
\end{proposition}
\begin{proof}
  There exist correspondences $f \in \Hom(\h_1(J)(i),(X,p))$ and $g
  \in \Hom((X,p),\h_1(J)(i))$ such that $g\circ f = \id_{\h_1(J)(i)}$
  and $f \circ g = p$. The correspondence $g \circ L \circ {}^tg \in
  \End(\h_1(J))$ induces an automorphism of $H_{1}(J)$ and, hence, is
  itself an automorphism (indeed by \cite[Prop. 4.5]{Scholl} we have
  $\End(\h_1(J)) = \End_k(J) \otimes \Q$ and it is well-known that a
  map between abelian varieties which induces an isomorphism in degree
  one homology must be an isogeny). Therefore, it admits an inverse
  $\alpha \in \End(\h_1(J))$. It is now straightforward to check that
  ${}^tp \circ {}^tg \circ \alpha \circ g \circ p$ is the inverse of
  $p \circ L \circ {}^tp$.
\end{proof}

As already proved by Scholl \cite{Scholl}, every smooth projective
variety satisfies the motivic Lefschetz conjecture in degrees $\leq
1$.

\begin{theorem}
  Let $f : X \r S$ be a dominant morphism defined over a field $k$
  from a smooth projective variety $X$ to a smooth projective surface
  $S$ such that the general fibre of $f_\Omega$ has trivial Chow group
  of zero-cycles.  Then $X$ satisfies the motivic Lefschetz
  conjecture in degrees $\leq 3$.  In particular, if $X$ has dimension
  $\leq 4$, then $X$ satisfies the motivic Lefschetz conjecture and
  hence the Lefschetz standard conjecture.
\end{theorem}
\begin{proof}
  By Theorem \ref{theoremCK}, $p_0$ factors through a
  point, and $p_1$ and $p_3$ factor through the $\h_1$ of a curve. The
  hard Lefschetz theorem says that the map $H_{2d-i}(X) \r H_i(X)$
  induced by intersecting $d-i$ times with a smooth hyperplane section
  is an isomorphism.  Therefore, the two propositions above give the
  motivic Lefschetz conjecture in degrees $0$, $1$ and $3$ for
  $X$.\medskip

  Let $\pi_2^{tr}$ be the idempotent of section 4.1. Let's prove that
  $L : (X,{}^t\pi_2^{tr},0) \r (X, \pi_2^{tr}, d-2)$ is an isomorphism
  of Chow motives. Because $\iota : H \r X$ is a linear section of $X$
  of dimension $2$, Proposition \ref{surjective} gives a non-zero
  integer $m$ such that $\Gamma_f \circ L \circ {}^t\Gamma_f = m \cdot
  \Delta_S$. It is then straightforward to check that $\frac{1}{m}
  \cdot {}^t\pi_2^{tr} \circ {}^t\Gamma_f \circ \Gamma_f \circ
  \pi_2^{tr}$ is the inverse of $ \pi_2^{tr} \circ L \circ {}^t
  \pi_2^{tr}$.

  Let $\pi_2^{alg}$ be the idempotent of section 4.3.  Because $L :
  (X,{}^t\pi_2^{tr},0) \r (X, \pi_2^{tr}, d-2)$ is an isomorphism and
  because $L_* : H_{2d-2}(X) \r H_2(X)$ is an isomorphism by the hard
  Lefschetz theorem, we see that $L$ induces an isomorphism $L :
  H_{2d-2}(X,{}^t\pi_2^{alg}) \r H_2(X,\pi_2^{alg})$. Proposition
  \ref{point} then shows that $ \pi_2^{alg} \circ L \circ {}^t
  \pi_2^{alg} \in \Hom((X,{}^t\pi_2^{alg}),(X,\pi_2^{alg},d-2))$ is an
  isomorphism.

  We have thus showed that $ \pi_2 \circ L \circ {}^t \pi_2 \in
  \Hom((X,{}^t\pi_2),(X,\pi_2,d-2))$ is an isomorphism. Since, by
  Theorem \ref{GS}, we know that $(X,p_2,d-2)\simeq (X,\pi_2,d-2)$ and
  $(X,{}^tp_2) \simeq (X,{}^t\pi_2)$, we get that $(X,p_2,d-2)$ is
  isomorphic to $(X,{}^tp_2)$. However, the isomorphism is induced by
  $p_2 \circ \pi_2 \circ L \circ {}^t \pi_2 \circ {}^t p_2 $ which is
  not quite the isomorphism we were aiming at.  \medskip

  By Remark \ref{missingorth}, it can be checked that in our
  particular setting we have $p_2 \circ \pi_2 = p_2$ so that $ p_2
  \circ L \circ {}^t p_2$ is an isomorphism with inverse $\frac{1}{m}
  \cdot {}^tp_2 \circ {}^t\pi_2 \circ {}^t\Gamma_f \circ \Gamma_f
  \circ \pi_2 \circ p_2$.

  Let's however give another proof that $ p_2 \circ L \circ {}^t p_2$
  is an isomorphism that might be useful in other situations.  We can
  conclude that $ p_2 \circ L \circ {}^t p_2$ is an isomorphism if we
  can show that it is equal to $ p_2 \circ \pi_2 \circ L \circ {}^t
  \pi_2 \circ {}^t p_2$. For this purpose, after examining the formula
  of Lemma \ref{linalg} defining $p_2$, it is enough to check that, for
  all correspondences $\alpha \in \CH_{2}(X \times X)$, we have $\pi_r
  \circ \alpha \circ {}^t \pi_2 = 0$ for $r = 0,1$. This is recorded
  in the lemma below.
\end{proof}

\begin{lemma}
  $\Hom ((X,{}^t\pi_2),(X,\pi_r,d-2)) = 0$ for $r=0$ or $1$.
\end{lemma}
\begin{proof}
  When $r=0$, $\pi_0 \circ \alpha \circ {}^t \pi_2^{alg}$ factors
  through a correspondence $\gamma \in \CH_1(P_1 \times P_0)$ for some
  zero-dimensional $P'$ and is thus zero for dimension reasons and
  $\pi_0 \circ \alpha \circ {}^t \pi_2^{tr}$ factors through a
  correspondence $\gamma \in \CH^0(S \times P_0)$ with $\gamma^* z =0$
  for any $z \in \CH_0(P)$.  Lemma \ref{trivialaction} then shows that
  $\gamma =0$ and hence $\pi_0 \circ \alpha \circ {}^t \pi_2^{tr} =
  0$.  When $r=1$, on the one hand, we have that $\pi_1 \circ \alpha
  \circ {}^t\pi_2^{alg}$ factors through a correspondence $\gamma \in
  \CH^0(P_1 \times C_0)$ with $\gamma^* z =0$ for any zero-cycle $z$ on
  $C_0$.  Lemma \ref{trivialaction} then shows that $\gamma =0$ and
  hence $\pi_1 \circ \alpha \circ {}^t\pi_2^{alg} = 0$. On the other
  hand, $\pi_1 \circ \alpha \circ {}^t\pi_2^{tr}$ factors through a
  correspondence $\gamma \in \CH^1(S \times C_0)$ with $\gamma^* z =0$
  for any zero-cycle $z$ on $(C_0)_\Omega$ and $\gamma_* z'$ for any
  zero-cycle $z'$ on $S_\Omega$ by functoriality of the Abel--Jacobi
  map.  Therefore thanks to Lemma \ref{2way}, we get $\gamma=0$ and
  hence $\pi_1 \circ \alpha \circ {}^t\pi_2^{tr}=0$.
\end{proof}

\begin{remark}
  The results of this section actually show that, for $X$ as in the
  theorem above and for the idempotents $p_i$ constructed in \S
  \ref{CK}, the map $L : (X,p_{2d-i}) \r (X,p_i,d-i)$ is an isomorphism
  for $i \leq 3$ for any choice of a smooth linear section $\iota : H
  \hookrightarrow X$ of dimension $i$.
\end{remark}

\section{Murre's conjectures for $X$} \label{Murreconj}

As shown by Jannsen \cite{Jannsen}, Murre's conjectures \cite{Murre1}
are equivalent to Bloch and Beilinson's. Let's recall them : \medskip

(A) $X$ has a Chow--K\"unneth decomposition $\{\pi_0, \ldots,
\pi_{2d}\}$ : There exist mutually orthogonal idempotents $\pi_0,
\ldots, \pi_{2d} \in \CH_d(X \times X)$ adding to the identity such
that $(\pi_i)_*H_*(X)=H_i(X)$ for all $i$.

(B) $\pi_0, \ldots, \pi_{2l-1},\pi_{d+l+1}, \ldots, \pi_{2d}$ act
trivially on $\CH_l(X)$ for all $l$.

(C) $F^i\CH_l(X) := \ker(\pi_{2l}) \cap \ldots \cap \ker(\pi_{2l+i-1})$
doesn't depend on the choice of the $\pi_j$'s. Here the $\pi_j$'s are
acting on $\CH_l(X)$.

(D) $F^1\CH_l(X) = \CH_l(X)_\hom$. \\

Before we consider Murre's conjectures for $X$ as in Theorem
\ref{mainth}, let's consider the following situation. Let $\Pi \in
\CH_d(X\times X) = \End(\h(X))$ be an idempotent which factors as
$$\h(X) \stackrel{g}{\longrightarrow} \h(Y)
\stackrel{f}{\longrightarrow} \h(X)$$ where $Y$ is a smooth projective
variety of dimension $\leq l+1$. (Actually, up to replacing $Y$ with
$Y \times \P^{l+1-\dim Y}$, we can assume $\dim Y =l+1$.) The
arguments in the proof of Propositions \ref{MurreB} and \ref{MurreC}
below are essentially contained in \cite{Vial2} and \cite{Vial3}.

\begin{proposition} \label{MurreB} Let $\Pi$ be as above.

  $\bullet$ If $\Pi_*H_{2l+1}(X) = 0$, then $\Pi$ acts trivially on
  $\CH_l(X)_\hom$.

  $\bullet$ If, moreover, $\Pi_*H_{2l}(X) = 0$, then $\Pi$ acts
  trivially on $\CH_l(X)$.
\end{proposition}
\begin{proof} Because $\Pi$ is an idempotent, we see that $g \circ f$
  acts trivially on $H^1(Y)$. Therefore, $g \circ f$ acts trivially on
  $\CH^1(Y)_\hom$. Thus, $\Pi$ acts trivially on $\CH_l(X)_\hom$. If,
  moreover, $\Pi_*H_{2i}(X) = 0$, then $g \circ f$ acts trivially on
  $\CH^1(Y)$. Thus, $\Pi$ acts trivially on $\CH_l(X)$.
\end{proof}

\begin{proposition} \label{MurreC} Let $\Pi$ be as above.

  $\bullet$ If $\Pi_*H_{*}(X) = H_{2l}(X)$, then $\ker \big( \Pi_* :
  \CH_l(X) \r \CH_l(X)\big) = \CH_l(X)_\hom$.

  $\bullet$ Assume that $k \subseteq \C$. If $\Pi_*H_{*}(X) =
  H_{2l+1}(X)$, then $\ker \big( \Pi_* : \CH_l(X)_\hom \r \CH_l(X)_\hom
  \big) = \ker \big( AJ_l : \CH_l(X)_\hom \r J_l(X) \otimes \Q \big)$.
  Here $AJ_l$ is the Abel--Jacobi map.
\end{proposition}
\begin{proof} In the first case, the inclusion $\subseteq$ follows
  immediately from the functoriality of the cycle class map with
  respect to the action of correspondences. The reverse inclusion
  $\supseteq$ follows from the first point of Proposition
  \ref{MurreB}.

  In the second case, we consider the Abel--Jacobi map $AJ_l :
  \CH_l(X)_\hom \r J_l(X) \otimes \Q$ instead of the cycle class map
  $\CH_l(X) \r H_{2l}(X)$. The Abel--Jacobi map is functorial with
  respect to the action of correspondences and, if a correspondence
  $\alpha \in \End(\h(X))$ induces the identity on $H_{2l+1}(X)$, then
  $\alpha_*$ induces the identity on $J_l(X) \otimes \Q$. This yields
  a commutative diagram
  \begin{center} $ \xymatrix{ \CH_l(X)_\hom \ar[d]^{AJ_l} \ar[r]^{g_*}
      & \CH^1(Y)_\hom \ar[d] \ar[r]^{f_*} & \CH_l(X)_\hom
      \ar[d]^{AJ_l}  \\
      J_{l}(X)(\C)\otimes \Q \ar[r] & \mathrm{Pic}^0_{Y}(\C)\otimes \Q
      \ar[r] & J_{l}(X)(\C)\otimes \Q.}$ \end{center} where the
  composite of the two bottom arrows is the identity and where the
  middle vertical arrow is injective. It is then straightforward to
  conclude by a simple diagram chase.
  \end{proof}

\begin{definition} \label{special} A smooth projective variety $X$ of
  dimension $d$ is said to have a \emph{special} Chow--K\"unneth
  decomposition $\{\pi_i\}_{0 \leq i \leq 2d}$ if, for all $i$,

  $\bullet$ $\pi_{2i}$ factors through a surface, i.e.  there is a
  surface $S_i$ such that $(X,\pi_{2i})$ is a direct summand of
  $\h(S_i)(i-1)$.

  $\bullet$ $\pi_{2i+1}$ factors through a curve, i.e.  there is a
  curve $C_i$ such that $(X,\pi_{2i+1})$ is a direct summand of
  $\h_1(C_i)(i)$.
\end{definition}

\begin{proposition} \label{specialMurre} Let $X$ be a smooth
  projective variety that has a special Chow--K\"unneth decomposition.
  Then homological and algebraic equivalence agree on $X$, and $X$
  satisfies Murre's conjectures (A), (B) and (D).  Moreover, if $ k
  \subseteq \C$, then the filtration $F$ on $\CH_l(X)$ does not depend
  on the choice of a special Chow--K\"unneth decomposition for $X$.
\end{proposition}
\begin{proof} That homological and algebraic equivalence agree on $X$
  follows from the well-known fact that they agree on zero-cycles and
  on codimension-one cycles.  That $X$ satisfies Murre's conjectures
  (B) is obvious and that $X$ satisfies (D) follows from the first
  point of Proposition \ref{MurreC}. That the induced filtration on
  the Chow groups of $X$ is independent of the choice of a special
  Chow--K\"unneth decomposition for $X$ is contained in Proposition
  \ref{MurreC}.
\end{proof}

Since the Chow--K\"unneth decomposition of $X$ as in Theorem
\ref{theoremCK} is a special Chow--K\"unneth decomposition, we can
state the following.

\begin{theorem} \label{Murre-theorem} Let $f : X \r S$ be a dominant
  morphism defined over a field $k$ between a smooth projective
  variety $X$ of dimension $\leq 4$ and a smooth projective surface
  $S$ such that the general fibre of $f_\Omega$ has trivial Chow group
  of zero-cycles. Then $X$ has a special Chow--K\"unneth decomposition
  which is self-dual and which satisfies Murre's conjectures (B) and
  (D) and, if $ k \subseteq \C$, then the induced filtration $F$
  on $\CH_l(X)$ does not depend on the choice of a special
  Chow--K\"unneth decomposition for $X$. Finally, whichever the
  characteristic of $k$ is, if $X$ is Kimura finite-dimensional, then
  $F$ does not depend on the choice of a Chow--K\"unneth decomposition
  for $X$.
\end{theorem}
\begin{proof}
  Theorem \ref{theoremCK} says that $X$ has a self-dual
  Chow--K\"unneth decomposition $\{p_i : 0 \leq i \leq 2d\}$ which is
  special. We may then conclude with Proposition \ref{specialMurre}
  that $X$ satisfies Murre's conjectures (B) and (D) and that, if
  $\mathrm{char} \ k=0$, then the induced filtration $F$ on $\CH_l(X)$
  does not depend on the choice of a special Chow--K\"unneth
  decomposition for $X$. When $X$ is Kimura finite-dimensional,
  Murre's conjecture (C) follows from applying \cite[Proposition
  3.1]{Vial2} to $X$ endowed with the Chow--K\"unneth decomposition
  given by the $p_i$'s.
\end{proof}

\section{Murre's conjectures for $X \times C$} \label{Murreconj2}

Let $f : X \r S$ be a dominant morphism from a smooth projective
fourfold to a smooth projective surface such that the general fibre of
$f_\Omega$ has trivial Chow group of zero-cycles. Consider the
self-dual Chow--K\"unneth decomposition $\{p_i : 0 \leq i \leq 8\}$ of
$X$ given in Theorem \ref{theoremCK} which, by Proposition
\ref{specialMurre}, is a Murre decomposition.  Let $C$ be a smooth
projective curve and let $\{p_0^C, p_1^C, p_2^C\}$ be a self-dual
Chow--K\"unneth decomposition for $C$ as described in \cite{Scholl}.
Then the variety $X \times C$ has a self-dual Chow--K\"unneth
decomposition given by $q_l := \sum_{i+j=l} p_i \times p_j^C$.  The
results of \cite{Vial2} make it possible to prove the following.

\begin{theorem}
  The fivefold $X \times C$ endowed with the above self-dual
  Chow--K\"unneth decomposition satisfies Murre's conjectures (A), (B)
  and (D).
\end{theorem}
\begin{proof} The idempotents $p_0$, $p_1$, $p_2$, $p_3$, $p_0^C$ and
  $p_1^C$ factor through varieties of respective dimension $0$,
  $1,2,1,0$ and $1$. Therefore, $(X,q_i)$ is isomorphic to a direct
  summand of the twisted motive of a surface for $i$ even. It is
  isomorphic to the direct summand of the twisted motive of a curve
  for $i = 1$ or $9$, and it is isomorphic to the direct summand of
  the twisted motive of a threefold for $i$ odd $\neq 1, 9$. Murre's
  conjectures (B) and (D) for $X$ endowed with the Chow--K\"unneth
  decomposition given by the $q_l$'s then follow immediately from
  Proposition \ref{MurreB} and from the first point of Proposition
  \ref{MurreC}.
\end{proof}

\section{Application to the finite-dimensionality problem}
\label{fdprob}

Kimura \cite{Kimura} introduced the notion of finite-dimensionality
for Chow motives. There he proved that any variety dominated by a
product of curves is finite-dimensional. It was proved by Guletskii
and Pedrini \cite{GP} that a surface with representable Chow group of
zero-cycles is Kimura finite-dimensional.  Gorchinskiy and Guletskii
\cite{GG} then proved that a threefold with representable Chow group
of zero-cycle is Kimura finite-dimensional.  This was subsequently
generalised to varieties of any dimension in \cite{Vial3} and to pure
motives in \cite{Vial1}. In their paper, Gorchinskiy and Guletskii
also prove \cite[Theorem 15]{GG} that, when $X$ is fibred over a curve
by del Pezzo or Enriques surfaces over an algebraically closed field
of characteristic zero, then $X$ has representable Chow group of
zero-cycles. Their method involves looking at the singular fibres of
the family.  Our Theorem \ref{isogeny} is more general and immediately
gives

\begin{theorem} \label{surface-Kimura} Let $X$ be a smooth projective
  threefold over a field $k$ and let $f : X \r C$ be a dominant
  morphism over a curve $C$ such that the general fibre of $f_\Omega$
  has trivial Chow group of zero-cycles. Then $X$ has representable
  Chow group of zero-cycles and is finite-dimensional in the sense of
  Kimura.
\end{theorem}

Godeaux surfaces are examples of surfaces of general type with trivial
Chow group of zero-cycles \cite{Godeaux}. Therefore, new cases
encompassed by the above theorem are given by threefolds fibred by
Godeaux surfaces over a curve. Let's make this more precise.

Let $\zeta$ be a primitive fifth root of unity. The group $G = \Z/5\Z$
acts on the complex projective space $\P_\C^3$ in the following way :
$\zeta \cdot [x_0:x_1:x_2:x_3] = [x_0:\zeta x_1: \zeta^2 x_2:
\zeta^3x_3]$. Let $\bar{V} := H^0(\P^3,O_{\P^3}(5))^G$ be the subspace
of $H^0(\P^3,O_{\P^3}(5))$ consisting of elements invariant under the
action of $G$ and let $V \hookrightarrow \bar{V}$ be the Zariski open
subset of $\bar{V}$ consisting of elements defining smooth quintic
surfaces. The monomials $X_i^5$ belong to $\bar{V}$ so that the
dimension of $\bar{V}$ is at least $4$. If $Y_v$ is a smooth quintic
in $\P^3$ given by the equation $v \in V$, then a local computation
shows that $Y_v$ cannot contain the fixed points of the action of $G$
on $\P^3$, so that the action of $G$ restricts to a free action on
$Y_v$. The quotient space $X_v := Y_v/G$ is a smooth projective
surface called a Godeaux surface.  These Godeaux surfaces fit into a
family $X \r \P (\bar{V})$.

Let's consider a smooth projective curve $C$ in $\P(\bar{V})$ that
meets $\P(\bar{V} - V)$ transversely. Then ${X}$ restricted to $C$
gives a smooth projective threefold ${X}|_C \r C$ with general fibre a
Godeaux surface. If $C$ is of general type ($g(C) \geq 2$), then, by a
result of Viehweg \cite{Viehweg} which is a special instance of the
Iitaka conjecture, ${X}|_C$ is a threefold of general type. We
have thus exhibited new examples of threefolds of general type with
representable Chow group of zero-cycles (obvious examples are given by
the product of a curve of general type with a Godeaux surface). Such
threefolds are also Kimura finite-dimensional thanks to \cite{GG}.
\medskip

In the following theorem, by conic fibration, we mean a dominant
morphism $X \r S$ whose general fibre is a conic.

\begin{theorem} \label{conic-Kimura} Let $X$ be a smooth projective
  threefold which is a conic fibration over a surface $S$ which is
  Kimura finite-dimensional. Then $X$ is finite-dimensional in the
  sense of Kimura.
\end{theorem}
\begin{proof}
  In section \ref{CK}, we proved that there is an orthogonal
  decomposition of the diagonal $\Delta_X = p_0 +p_1+p_2^{tr} +
  p_2^{alg} + p_3 + {}^tp_2^{alg}+ {}^tp_2^{tr}+{}^tp_1+{}^tp_0$ with
  $(X,p_0)$ and $(X,p_2^{alg})$ isomorphic to twisted motives of
  points, $(X,p_1)$ and $(X,p_3)$ isomorphic to direct summands of
  twisted motives of curves; and with $(X,p_2^{tr})$ isomorphic to
  $(S,\pi_2^{tr,S})$. Motives of points and motives of curves are
  finite-dimensional \cite{Kimura}. Since $S$ is Kimura
  finite-dimensional by assumption and since finite-dimensionality is
  stable under direct summand \cite{Kimura}, we have that
  $(X,p_2^{tr})$ is finite-dimensional.  Therefore $X$ is Kimura
  finite-dimensional.
\end{proof}

\begin{theorem} \label{C-Murre2}
  Let $X$ be as in Theorem \ref{surface-Kimura} or as in Theorem
  \ref{conic-Kimura}. Then $X$ satisfies Murre's conjectures (A),
  (B), (C) and (D).
\end{theorem}
\begin{proof}
  By Theorem \ref{Murre-theorem}, $X$ is Kimura finite-dimensional and
  has a Chow--K\"unneth decomposition that satisfies Murre's
  conjectures (B) and (D). According to \cite[Theorem 4.8]{Vial2}, we
  can conclude that $X$ satisfies Murre's conjecture (C) if the
  cohomology of $X$ is generated via the action of correspondences by
  the cohomology of surfaces. By Theorem \ref{Murre-theorem} again,
  $X$ satisfies the Lefschetz standard conjecture. Therefore, it
  suffices to show that $H_3(X)$ is generated by the $H_1$ of a curve.
  But then, this follows, via Bloch--Srinivas \cite{BS}, from the fact
  that $\CH_0(X_\Omega)$ is supported on a surface.
\end{proof}

\section{A fourfold of general type satisfying Murre's conjectures}
\label{general}

In this section we wish to give explicit examples of fourfolds
satisfying the assumptions of Theorem \ref{mainth}. For this purpose,
we consider two-dimensional families of surfaces having trivial Chow
group of zero-cycles.

A first type of such families was already given in Theorem
\ref{ratconn2} and consisted in separably rationally connected
fibrations over a surface (separably rationally connected is the same
as rationally connected if the base field has characteristic zero).
Precisely, Theorem \ref{ratconn2} considered smooth projective
varieties $X$ over a field $k$ with an equidimensional map $X \r S$ to
a smooth projective surface with general fibre being separably
rationally connected.
However, such a fourfold is not of general type. A natural question is
to ask whether it is possible to construct a fourfold of general type
that has a self-dual Murre decomposition. In \cite[\S 2.3 \& Cor
4.12]{Vial2}, we produced examples of such fourfolds. These fourfolds
have the property of having trivial Chow group of zero-cycles. Obvious
examples were given by the product of two surfaces of general type
with trivial Chow group of zero-cycles (e.g. Godeaux surfaces).
Another example, a fourfold of Godeaux type, was given. The strategy
consisted in checking the validity of the generalised Hodge conjecture
for this fourfold.

We are now going to give an example of fourfold of general type with
non-trivial (and in fact non-representable) Chow group of zero-cycles
that has a self-dual Murre decomposition. Let's take up the notations
of the previous section and let's consider the family $X \r
\P(\bar{V})$. Let then $S$ be a high-degree (i.e. $\geq 5$) complete
intersection which is a smooth surface in $\P(\bar{V})$ meeting
$\P(\bar{V}-V)$ transversely. Then $X|_S$ is a projective fourfold
with $X|_S \r S$ having a smooth Godeaux surface as a general fibre. A
desingularization $X' \r X|_S$ gives a morphism $X' \r S$ with general
fibre being of general type and having trivial Chow group of
zero-cycles. This is because these two conditions are birational
invariants. The high-degree condition on $S$ imposes that $S$ is of
general type and has non-representable $\CH_0(S)_\alg$. Therefore, by
Viehweg's result \cite{Viehweg}, $X'$ is of general type; and, by
Theorem \ref{mainth}, $X'$ has a self-dual Murre decomposition which
satisfies the motivic Lefschetz conjecture.

\section{Application to unramified cohomology} \label{unram}

Following the fundamental result of Colliot-Th\'el\`ene, Sansuc and
Soul\'e \cite{CTSS} which asserts that the degree-three unramified
cohomology groups $H^3_{nr}(S/k,\Q_l/\Z_l(2))$ vanish for all prime
numbers $l$ for $S$ a smooth projective surface defined over a field
$k$ which is either finite or separably closed, it is proved in
\cite[Proposition 3.2]{CTK} that, if $X$ is a smooth projective variety
defined over a field $k$ which is either finite or separably closed
such that its Chow group of zero-cycles is supported on a surface,
then the groups $H^3_{nr}(X/k,\Q_l/\Z_l(2))$ are finite for all prime
numbers $l$ and vanish for almost all $l$.  Therefore, any variety $X$
defined over a finite field or a separably closed field such that its
restriction $X_\Omega$ to a universal domain $\Omega$ satisfies the
assumptions of Theorem \ref{isogeny} has finite degree-three
unramified cohomology $\bigoplus_l H^3_{nr}(X/k,\Q_l/\Z_l(2))$. In
particular, the fourfold of general type of section \ref{general}, when
defined over a finite field or a separably closed field, has finite
degree-three unramified cohomology. Furthermore, as a straightforward
application of Theorem \ref{repsurface}, we get

\begin{proposition} \label{bielliptic-unram} Let $f : X \r C$ be a
  dominant and generically smooth morphism from a smooth projective
  variety $X$ to a smooth projective curve $C$ defined over a field
  $k$ which is either finite or separably closed.  Assume that the
  general fibre $Y$ of $f_\Omega$ is such that $\CH_0(Y)_\alg$
  is representable.  Then $H^3_{nr}(X/k,\Q_l/\Z_l(2))$ is finite for
  all prime numbers $l$ and vanishes for almost all $l$.
\end{proposition}

Since unramified cohomology is a birational invariant for smooth
projective varieties, the conclusion of the above theorem still holds
for a smooth projective variety $X'$ which is birational to the
variety $X$ of the theorem.  For instance, we get finiteness of degree
three unramified cohomology for threefolds which are the smooth
compactification of one-dimensional families of smooth projective
surfaces defined over a finite field or a separably closed field whose
generic member is a bielliptic surface.

\begin{footnotesize}
  \bibliographystyle{plain} 
  \bibliography{bib} \medskip

\def\cprime{$'$}
\begin{thebibliography}{10}

\bibitem{BS}
S.~Bloch and V.~Srinivas.
\newblock Remarks on correspondences and algebraic cycles.
\newblock {\em Amer. J. Math.}, 105(5):1235--1253, 1983.

\bibitem{CTK}
J.-L. Colliot-Th\'el\`ene and B.~Kahn.
\newblock {Cycles de codimension $2$ et $H^3$ non ramifi\'e pour les
  vari\'et\'es sur les corps finis}.
\newblock {Preprint, arXiv:1104.3350v2}.

\bibitem{CTSS}
Jean-Louis Colliot-Th{\'e}l{\`e}ne, Jean-Jacques Sansuc, and Christophe
  Soul{\'e}.
\newblock Torsion dans le groupe de {C}how de codimension deux.
\newblock {\em Duke Math. J.}, 50(3):763--801, 1983.

\bibitem{dAMS}
Pedro~Luis del Angel and Stefan M{\"u}ller-Stach.
\newblock Motives of uniruled {$3$}-folds.
\newblock {\em Compositio Math.}, 112(1):1--16, 1998.

\bibitem{Fulton}
William Fulton.
\newblock {\em Intersection theory}, volume~2 of {\em Ergebnisse der Mathematik
  und ihrer Grenzgebiete. 3. Folge. A Series of Modern Surveys in Mathematics
  [Results in Mathematics and Related Areas. 3rd Series. A Series of Modern
  Surveys in Mathematics]}.
\newblock Springer-Verlag, Berlin, second edition, 1998.

\bibitem{GG}
S.~Gorchinskiy and V.~Guletskii.
\newblock {Motives and representability of algebraic cycles on threefolds over
  a field}.
\newblock {JAG, to appear}.

\bibitem{Gor}
Sergey Gorchinskiy.
\newblock {Letter to the author}.
\newblock Feb. 10, 2012.

\bibitem{GHM}
B.~Brent Gordon, Masaki Hanamura, and Jacob~P. Murre.
\newblock Absolute {C}how-{K}\"unneth projectors for modular varieties.
\newblock {\em J. Reine Angew. Math.}, 580:139--155, 2005.

\bibitem{GP}
V.~Guletski{\u\i} and C.~Pedrini.
\newblock Finite-dimensional motives and the conjectures of {B}eilinson and
  {M}urre.
\newblock {\em $K$-Theory}, 30(3):243--263, 2003.
\newblock Special issue in honor of Hyman Bass on his seventieth birthday. Part
  III.

\bibitem{Jannsen}
Uwe Jannsen.
\newblock Motivic sheaves and filtrations on {C}how groups.
\newblock In {\em Motives (Seattle, WA, 1991)}, volume~55 of {\em Proc. Sympos.
  Pure Math.}, pages 245--302. Amer. Math. Soc., Providence, RI, 1994.

\bibitem{KahnSujatha}
B.~Kahn and R.~Sujatha.
\newblock {Birational motives, I pure birational motives}.
\newblock {Preprint, February 27, 2009. arXiv:0902.4902}.

\bibitem{KMP}
Bruno Kahn, Jacob~P. Murre, and Claudio Pedrini.
\newblock On the transcendental part of the motive of a surface.
\newblock In {\em Algebraic cycles and motives. Vol. 2}, volume 344 of {\em
  London Math. Soc. Lecture Note Ser.}, pages 143--202. Cambridge Univ. Press,
  Cambridge, 2007.

\bibitem{Kimura}
Shun-Ichi Kimura.
\newblock Chow groups are finite dimensional, in some sense.
\newblock {\em Math. Ann.}, 331(1):173--201, 2005.

\bibitem{Kleiman}
Steven~L. Kleiman.
\newblock The standard conjectures.
\newblock In {\em Motives ({S}eattle, {WA}, 1991)}, volume~55 of {\em Proc.
  Sympos. Pure Math.}, pages 3--20. Amer. Math. Soc., Providence, RI, 1994.

\bibitem{Kollar}
J{\'a}nos Koll{\'a}r.
\newblock {\em Rational curves on algebraic varieties}, volume~32 of {\em
  Ergebnisse der Mathematik und ihrer Grenzgebiete. 3. Folge. A Series of
  Modern Surveys in Mathematics}.
\newblock Springer-Verlag, Berlin, 1996.

\bibitem{Manin}
Ju.~I. Manin.
\newblock Correspondences, motifs and monoidal transformations.
\newblock {\em Mat. USSR-Sb.}, 6:439--470, 1968.

\bibitem{MSS}
S.~M\"uller-Stach and M.~Saito.
\newblock {Relative Chow-K\"unneth decompositions for morphisms of threefolds}.
\newblock {to appear in Crelle's journal}.

\bibitem{Murre}
J.~P. Murre.
\newblock On the motive of an algebraic surface.
\newblock {\em J. Reine Angew. Math.}, 409:190--204, 1990.

\bibitem{Murre1}
J.~P. Murre.
\newblock On a conjectural filtration on the {C}how groups of an algebraic
  variety. {I}. {T}he general conjectures and some examples.
\newblock {\em Indag. Math. (N.S.)}, 4(2):177--188, 1993.

\bibitem{Scholl}
A.~J. Scholl.
\newblock Classical motives.
\newblock In {\em Motives (Seattle, WA, 1991)}, volume~55 of {\em Proc. Sympos.
  Pure Math.}, pages 163--187. Amer. Math. Soc., Providence, RI, 1994.

\bibitem{Vial2}
Charles Vial.
\newblock {Niveau and coniveau filtrations on cohomology groups and Chow
  groups}.
\newblock {\em {Proc. Lond. Math. Soc.}}
\newblock {to appear}.

\bibitem{Vial3}
Charles Vial.
\newblock {Projectors on the intermediate algebraic Jacobians}.
\newblock {Preprint}.

\bibitem{Vial1}
Charles Vial.
\newblock Pure motives with representable {C}how groups.
\newblock {\em C. R. Math. Acad. Sci. Paris}, 348(21-22):1191--1195, 2010.

\bibitem{Viehweg}
Eckart Viehweg.
\newblock Weak positivity and the additivity of the {K}odaira dimension for
  certain fibre spaces.
\newblock In {\em Algebraic varieties and analytic varieties ({T}okyo, 1981)},
  volume~1 of {\em Adv. Stud. Pure Math.}, pages 329--353. North-Holland,
  Amsterdam, 1983.

\bibitem{Godeaux}
Claire Voisin.
\newblock Sur les z\'ero-cycles de certaines hypersurfaces munies d'un
  automorphisme.
\newblock {\em Ann. Scuola Norm. Sup. Pisa Cl. Sci. (4)}, 19(4):473--492, 1992.

\end{thebibliography}

  \textsc{DPMMS, University of Cambridge, Wilberforce Road, Cambridge,
    CB3 0WB, UK}
  \end{footnotesize}

  \textit{e-mail :} \texttt{c.vial@dpmms.cam.ac.uk}

\end{document}